\documentclass[reqno,12pt]{amsart}
                            
\usepackage{color}
\usepackage{mathrsfs}
\usepackage{MnSymbol}
\usepackage{graphicx}
\usepackage{caption}
\usepackage{subcaption}
\usepackage{todonotes} 

\DeclareFontFamily{OT1}{pzc}{}
\DeclareFontShape{OT1}{pzc}{m}{it}{<-> s * [1.050] pzcmi7t}{}
\DeclareMathAlphabet{\mathpzc}{OT1}{pzc}{m}{it}

\newcommand{\ii}{\hat{\imath}}

\newcommand{\kh}{k^h}
\newcommand{\dive}{\mathop{\vec\nabla \cdot}}
\newcommand{\curl}{\mathop{\vec\nabla \times}}
\newcommand{\scurl}{\mathop{\nabla \times}}
\newcommand{\grad}{\ensuremath{{\vec\nabla}}}
\newcommand{\vu}{\vec u}
\newcommand{\vv}{\vec v}
\newcommand{\vr}{\vec r}
\newcommand{\vw}{\vec w}
\newcommand{\vn}{\vec n}
\newcommand{\vt}{\vec t}
\newcommand{\vE}{\vec E}
\newcommand{\vH}{\vec H}
\newcommand{\vj}{\vec{\jmath}}
\newcommand{\vl}{\vec l}
\newcommand{\vp}{{\vec{p}}}
\newcommand{\vcalU}{\vec{\mathcal{U}}}
\newcommand{\vcalE}{\vec{\mathcal{E}}}
\newcommand{\vcalH}{\vec{\mathcal{H}}}
\newcommand{\vcalr}{\vec{\mathpzc{r}}}
\newcommand{\calE}{\mathcal{E}}
\newcommand{\calH}{\mathcal{H}}
\newcommand{\ip}[1]{\langle {#1} \rangle}
\newcommand{\im}{\operatorname{Im}}
\newcommand{\re}{\operatorname{Re}}

\newcommand{\veps}{\varepsilon}
\newcommand{\Th}{\mathcal{T}_h}
\newcommand{\ZZZ}{\mathbb{Z}}
\def\d{\partial}
\renewcommand{\geq}{\geqslant}
\renewcommand{\ge}{\geqslant}
\renewcommand{\leq}{\leqslant}
\renewcommand{\le}{\leqslant}

\newtheorem{theorem}{Theorem}
\newtheorem{lemma}{Lemma}

\newcommand{\thx}{JG and NO were supported in part by the NSF grant
  DMS-1318916 and the AFOSR grant FA9550-12-1-0484. NO gratefully
  acknowledges support in the form of an INRIA internship where
  discussions leading to this work originated.  
  All authors wish to thank
  INRIA Sophia Antipolis
  M\'editerran\'ee for hosting the authors there and 
  facilitating this research.
}

\newcommand{\AbsBackground}{Simulation of wave propagation through
  complex media relies on proper understanding of the properties of
  numerical methods when the wavenumber is real and complex.}

\newcommand{\AbsMethods}{Numerical methods of the Hybrid Discontinuous
  Galerkin (HDG) type are considered for simulating waves that satisfy
  the Helmholtz and Maxwell equations. It is shown that these methods,
  when wrongly used, give rise to singular systems for complex
  wavenumbers.}

\newcommand{\AbsResults}{A sufficient condition on the HDG
  stabilization parameter for guaranteeing unique solvability of the
  numerical HDG system, both for Helmholtz and Maxwell systems, is
  obtained for complex wavenumbers.  For real wavenumbers, results
  from a dispersion analysis are presented. An asymptotic expansion of
  the dispersion relation, as the number of mesh elements per wave
  increase, reveal that some choices of the stabilization parameter
  are better than others.}

\newcommand{\AbsConclusions}{To summarize the findings, there are
  values of the HDG stabilization parameter that will cause the HDG
  method to fail for complex wavenumbers. However, this failure is
  remedied if the real part of the stabilization parameter has the
  opposite sign of the imaginary part of the wavenumber.  When the
  wavenumber is real, values of the stabilization parameter that
  asymptotically minimize the HDG wavenumber errors are
  found on the imaginary axis. Finally, a dispersion analysis of 
  the mixed hybrid Raviart-Thomas method showed that its wavenumber errors
  are an order smaller than those of the HDG method.}

\title[Stabilization and wavenumber in HDG]{Stabilization in relation to wavenumber in HDG methods}

\author{J.~Gopalakrishnan}
\address{Portland State University, PO Box 751, Portland, OR 97207-0751, USA}
\email{gjay@pdx.edu}

\author{S.~Lanteri}
\address{{INRIA Sophia Antipolis M\'editerran\'ee},
  2004 Route des Lucioles, BP 93, 06902 Sophia Antipolis Cedex, France}
\email{stephane.lanteri@inria.fr}

\author{N.~Olivares}
\address{Portland State University, PO Box 751, Portland, OR 97207-0751, USA}
\email{nmo@pdx.edu}

\author{R.~Perrussel}
\address{LAPLACE (LAboratoire PLAsma et Conversion d’Energie),
  Universit\'e de Toulouse, CNRS/INPT/UPS, Toulouse,{France}}
\email{perrussel@laplace.univ-tlse.fr}

\thanks{\thx}


\begin{document}

\begin{abstract}
  \AbsBackground\; \AbsMethods\; \AbsResults\; \AbsConclusions
\end{abstract}

\maketitle

\section*{Background}

Wave propagation through complex structures, composed of both
propagating and absorbing media, are routinely simulated using
numerical methods.  Among the various numerical methods used, the
Hybrid Discontinuous Galerkin (HDG) method has emerged as an
attractive choice for such simulations. The easy passage to high order
using interface unknowns, condensation of all interior variables,
availability of error estimators and adaptive algorithms, are some of
the reasons for the adoption of HDG methods.

It is important to design numerical methods that remain stable as the
wavenumber varies in the complex plane.  For example, in applications
like computational lithography, one finds absorbing materials with
complex refractive index in parts of the domain of simulation.
Other examples are furnished by
meta-materials.  A separate and important reason for requiring
such stability emerges in the computation of
resonances by iterative searches in the complex plane.  It is common
for such iterative algorithms to solve a source problem with a complex
wavenumber as its current iterate. Within such algorithms, if the HDG
method is used for discretizing the source problem, it is imperative
that the method remains stable for all complex wavenumbers.

One focus of this study is on complex wavenumber cases in acoustics
and electromagnetics, motivated by the above-mentioned examples.  Ever
since the invention of the HDG method in~\cite{CockbGopalLazar09}, it
has been further developed and extended to other problems in many
works (so many so that it is now impractical to list all references on
the subject here). Of particular interest to us are works that applied
HDG ideas to wave propagation problems such
as~\cite{CuiZhang13,GiorgFernaHuert13,GriesMonk11,HuertRocaAleks12,
  LiLantePerru13a,LiLantePerru14, NguyePeraiCockb11a}. We will make
detailed comparisons with some of these works in a later
section. However, none of these references address the stability
issues for complex wavenumber cases. While the choice of the HDG
stabilization parameter in the real wave number case can be safely
modeled after the well-known choices for elliptic
problems~\cite{CockbGopalSayas10}, the complex wave number case is
essentially different. This will be clear right away from a few
elementary calculations in the next section, which show that the
standard prescriptions of stabilization parameters are not always
appropriate for the complex wave number case.  This then raises
further questions on how the HDG stabilization parameter should be
chosen in relation to the wavenumber, which are addressed in later
sections.

Another focus of this study is on the difference in speeds of the
computed and the exact wave, in the case of real wavenumbers. By means
of a dispersion analysis, one can compute the discrete wavenumber of a
wave-like solution computed by the HDG method, for any given exact
wavenumber. An extensive bibliography on dispersion analyses for the
standard finite element method can be obtained
from~\cite{Ainsw04,DeraeBabusBouil99}. For nonstandard finite element
methods however, dispersion analysis is not so
common~\cite{GopalMugaOliva14}, and for the HDG method, it does not
yet exist. We will show that useful insights into the HDG method can
be obtained by a dispersion analysis. In multiple dimensions, the
discrete wavenumber depends on the propagation angle. Analytic
computation of the dispersion relation is feasible in the lowest order
case. We are thus able to study the influence of the stabilization
parameter on the discrete wavenumber and offer recommendations on
choosing good stabilization parameters. The optimal stabilization
parameter values are found not to depend on the wavenumber. In the
higher order case, since analytic calculations pose difficulties, we
conduct a dispersion analysis numerically.

We begin, in the next section, by describing the HDG methods. 
We set the stage for this study  by 
showing that the commonly chosen HDG stabilization parameter values
for elliptic problems are not appropriate for all complex
wavenumbers. In the subsequent section, we discover a constraint on
the stabilization parameter, dependent on the wavenumber, that
guarantees unique solvability of both the global and the local HDG
problems. Afterward, we perform a dispersion analysis for both the HDG method 
and a mixed method and discuss the results.

\section*{Methods of the HDG type}

We borrow the basic methodology for constructing HDG methods
from~\cite{CockbGopalLazar09} and apply it to the time-harmonic
Helmholtz and Maxwell equations (written as first order
systems). While doing so, we set up the notations used throughout,
compare the formulation we use with other  existing works, and show
that for complex wavenumbers there are stabilization parameters that
will cause the HDG method to fail.

\subsection*{Undesirable HDG stabilization parameters for the Helmholtz system}
\label{sec:helmholtz}

We begin by considering the lowest order HDG system for Helmholtz
equation. Let $k$ be a complex number. Consider the Helmholtz system
on $\Omega\subset\mathbb{R}^2$ with homogeneous Dirichlet boundary
conditions,
\begin{subequations}
  \label{eq:Helmholtz}
  \begin{align}
    \ii k \vcalU + \grad \varPhi & = \vec 0, \qquad \text{in } \Omega,
    \\
    \ii k \varPhi + \dive \vcalU & = f, \qquad \text{in } \Omega,
    \\
    \varPhi &= 0, \qquad \text{on } \partial\Omega,
  \end{align}
\end{subequations}
where $f\in L^2(\Omega)$. Let $\Th$ denote a square or triangular mesh
of disjoint elements $K$, so $\overline\Omega=\cup_{K\in\Th}\overline
K$, and let $\mathcal{F}_h$ denote the collection of edges. The HDG
method produces an approximation $(\vu, \phi, \hat \phi)$ to the exact
solution $(\vcalU, \varPhi, \hat \varPhi)$, where $\hat\varPhi$
denotes the trace of $\varPhi$ on the collection of element boundaries
$\partial \Th$. The HDG solution $(\vu, \phi, \hat \phi)$ is in the
finite dimensional space $ V_{h} \times W_{h} \times M_{h}$ defined by
\begin{align*}
  V_h&=\{ \vv\in (L^2(\Omega))^2~:~\vv|_K\in V(K),~\forall K\in\Th\}\\
  W_h&=\{ \psi\in L^2(\Omega)~:~\psi|_K\in W(K),~\forall K\in\Th\}\\
  M_h&=\{ \hat\psi\in L^2( \mathop{\bigcup}_{F \in \mathcal{F}_h}
  F)~:~\hat\psi|_F\in M(F),~\forall F\in\mathcal{F}_h \text{ and
  }\hat\psi|_{\d\Omega}=0\},
\end{align*}
with polynomial spaces $V(K)$, $W(K)$, and $M(F)$ specified
differently depending on element type:
\begin{align*}
  &\hspace{-0.6cm}\text{\underline{Triangles}}\qquad\qquad&&\hspace{-0.5cm}\text{\underline{Squares}}
  \\
  V(K)&=(\mathcal{P}_p(K))^2          &V(K)&=(\mathcal{Q}_p(K))^2\\
  W(K)&=\mathcal{P}_p(K)                     & W(K)&=\mathcal{Q}_p(K) \\
  M(F)&=\mathcal{P}_p(F) &M(F)&=\mathcal{P}_p(F).
\end{align*}
Here, for a given domain $D$, $\mathcal{P}_p(D)$ denotes polynomials
of degree at most $p$, and $\mathcal{Q}_p(D)$ denotes polynomials of
degree at most $p$ in each variable.

The HDG solution satisfies
\begin{subequations}
  \label{eq:globalHelmholtz}
  \begin{align}
    & \sum_{K \in \Th} & \ii k ( \vu, \vv)_K - (\phi, \dive \vv)_K +
    \ip{\hat \phi, \vv\cdot\vn}_{\d K} & = 0, 
    \\
    & \sum_{K\in\Th} & -(\dive \vu, \psi)_K + \ip{ \tau \hat \phi,
      \psi}_{\d K} -\ip{\tau \phi, \psi}_{\d K} - \ii k (\phi,\psi)_K
    & = -(f,\psi)_{\Omega}, 
    \\
    & \sum_{K\in\Th} & \ip{ \vu\cdot \vn + \tau ( \phi - \hat\phi),
      \hat \psi}_{\d K} & = 0,
  \end{align}
\end{subequations}
for all $\vv \in V_h$, $\psi \in W_{h}$, and $\hat\psi \in M_{h}.$
The last equation enforces the conservativity of the numerical flux
\begin{equation}
  \label{eq:numflux}
  \hat u \cdot \vn = 
  \vu\cdot \vn + \tau ( \phi - \hat\phi).
\end{equation}
The stabilization parameter $\tau$ is assumed to be constant on each
$\partial K$. We are interested in how the choice of $\tau$ in
relation to $k$ affects the method, especially when $k$ is complex
valued. Comparisons of this formulation with other HDG formulations
for Helmholtz equations in the literature are summarized in
Table~\ref{tab:comp}.

\begin{table}
  \centering
  \begin{tabular}{|c||c|c|}
    \hline  
    Reference & Their notations and equations & Connection to our formulation
    \\
    \hline   
    & & \\
    $\displaystyle{
      \begin{gathered}
        \text{\cite{CuiZhang13}}
        \\
        \text{Helmholtz case}
      \end{gathered}
    }$
    & 
    $\displaystyle{
      \begin{gathered}
        \vec{q}\raisebox{-1ex}{\tiny\cite{CuiZhang13}} + \grad u
        \raisebox{-1ex}{\tiny\cite{CuiZhang13}} = \vec 0
        \\
        \dive \vec q\raisebox{-1ex}{\tiny\cite{CuiZhang13}} - k^2
        u\raisebox{-1ex}{\tiny\cite{CuiZhang13}} = 0
        \\
        \hat q\raisebox{-1ex}{\tiny\cite{CuiZhang13}} \cdot \vec n =
        \vec q\raisebox{-1ex}{\tiny\cite{CuiZhang13}} \cdot \vec n +
        \ii \tau\raisebox{-1ex}{\tiny\cite{CuiZhang13}}
        (u\raisebox{-1ex}{\tiny\cite{CuiZhang13}}- \hat
        u\raisebox{-1ex}{\tiny\cite{CuiZhang13}} )
      \end{gathered}
    }$
    & 
    $\displaystyle{
      \begin{aligned}
        \tau\raisebox{-1ex}{\tiny\cite{CuiZhang13}} & = k\, \tau
        \\
        \ii k u\raisebox{-1ex}{\tiny\cite{CuiZhang13}} & = \phi
        \\
        \vec q\raisebox{-1ex}{\tiny\cite{CuiZhang13}} & =\vu
      \end{aligned}
    }$
    \\ 
    \hline
    & & \\
    $\displaystyle{
      \begin{gathered}
        \text{\cite{GriesMonk11}}
        \\
        \text{Helmholtz case}
      \end{gathered}
    }$
    & 
    $\displaystyle{
      \begin{gathered}
        \ii k \vec{q}\raisebox{-1ex}{\tiny\cite{GriesMonk11}} + \grad
        u \raisebox{-1ex}{\tiny\cite{GriesMonk11}} = \vec 0
        \\
        \ii k u\raisebox{-1ex}{\tiny\cite{GriesMonk11}} + \dive \vec
        q\raisebox{-1ex}{\tiny\cite{GriesMonk11}} = 0
        \\
        \hat q\raisebox{-1ex}{\tiny\cite{GriesMonk11}} \cdot \vec n =
        \vec q\raisebox{-1ex}{\tiny\cite{GriesMonk11}} \cdot \vec n +
        \tau\raisebox{-1ex}{\tiny\cite{GriesMonk11}}
        (u\raisebox{-1ex}{\tiny\cite{GriesMonk11}}- \hat
        u\raisebox{-1ex}{\tiny\cite{GriesMonk11}} )
      \end{gathered}
    }$
    & 
    $\displaystyle{
      \begin{aligned}
        \tau\raisebox{-1ex}{\tiny\cite{GriesMonk11}} & = \tau
        \\
        u\raisebox{-1ex}{\tiny\cite{GriesMonk11}} & = \phi
        \\
        \vec q\raisebox{-1ex}{\tiny\cite{GriesMonk11}} & =\vu
      \end{aligned}
    }$
    \\ 
    \hline
    & & \\
    $\displaystyle{
      \begin{gathered}
        \text{\cite{LiLantePerru13a}}
        \\
        \text{2D Maxwell case}
      \end{gathered}
    }$
    & 
    $\displaystyle{
      \begin{gathered} 
        \ii \omega \raisebox{-1ex}{\tiny\cite{LiLantePerru13a}}
        \varepsilon_r E \raisebox{-1ex}{\tiny\cite{LiLantePerru13a}}
        -\scurl \vH\raisebox{-1ex}{\tiny\cite{LiLantePerru13a}}= 0
        \\
        \ii \omega\raisebox{-1ex}{\tiny\cite{LiLantePerru13a}} \mu_r
        \vH\raisebox{-1ex}{\tiny\cite{LiLantePerru13a}} +\curl E
        \raisebox{-1ex}{\tiny\cite{LiLantePerru13a}}= \vec 0
        \\
        \hat H \raisebox{-1ex}{\tiny\cite{LiLantePerru13a}} =
        \vH\raisebox{-1ex}{\tiny\cite{LiLantePerru13a}}
        +\tau\raisebox{-1ex}{\tiny\cite{LiLantePerru13a}} (E
        \raisebox{-1ex}{\tiny\cite{LiLantePerru13a}} -\hat
        E\raisebox{-1ex}{\tiny\cite{LiLantePerru13a}})\vec t
      \end{gathered}
    }$
    & 
    $\displaystyle{
      \begin{aligned}
        \tau\raisebox{-1ex}{\tiny\cite{LiLantePerru13a}} &=
        \sqrt{\frac{\varepsilon_r}{\mu_r}}\tau \\
        \omega\raisebox{-1ex}{\tiny\cite{LiLantePerru13a}} &=
        \omega\sqrt{\varepsilon_0\mu_0}\\
        E\raisebox{-1ex}{\tiny\cite{LiLantePerru13a}} &=
        \frac{1}{\sqrt{\varepsilon_r}}E,~
        \vH\raisebox{-1ex}{\tiny\cite{LiLantePerru13a}} =
        \frac{1}{\sqrt{\mu_r}}\vH\\
      \end{aligned}
    }$
    \\ 
    \hline
    & & \\
    $\displaystyle{
      \begin{gathered}
        \text{\cite{NguyePeraiCockb11a}}
        \\
        \text{Maxwell case}
      \end{gathered}
    }$
    & 
    $\displaystyle{
      \begin{gathered}
        \mu \vec w\raisebox{-1ex}{\tiny\cite{NguyePeraiCockb11a}} -
        \curl \vu\raisebox{-1ex}{\tiny\cite{NguyePeraiCockb11a}} =
        \vec 0
        \\
        \curl \vw\raisebox{-1ex}{\tiny\cite{NguyePeraiCockb11a}}
        -\veps \omega^2
        \vu\raisebox{-1ex}{\tiny\cite{NguyePeraiCockb11a}} =\vec 0
        \\
        \hat w\raisebox{-1ex}{\tiny\cite{NguyePeraiCockb11a}} =
        \vw\raisebox{-1ex}{\tiny\cite{NguyePeraiCockb11a}} +
        \tau\raisebox{-1ex}{\tiny\cite{NguyePeraiCockb11a}} (
        \vu\raisebox{-1ex}{\tiny\cite{NguyePeraiCockb11a}} - \hat
        u\raisebox{-1ex}{\tiny\cite{NguyePeraiCockb11a}} ) \times \vn
      \end{gathered}
    }$
    & 
    $\displaystyle{
      \begin{aligned}
        \tau\raisebox{-1ex}{\tiny\cite{NguyePeraiCockb11a}} & = \ii
        \,\sqrt{ \frac{\varepsilon \omega^2}{\mu} } \, \tau
        \\
        \mu \vw\raisebox{-1ex}{\tiny\cite{NguyePeraiCockb11a}} & =-\ii
        k \vH,\text{ with }k = \omega \sqrt{\mu\varepsilon},
        \\
        \vu\raisebox{-1ex}{\tiny\cite{NguyePeraiCockb11a}} & =\vE
      \end{aligned}
    }$
    \\
    \hline
  \end{tabular}
  \caption{Comparison with some HDG formulations in other papers. 
    Notations in the indicated external references are used after
    subscripting them by the reference number. Notations without 
    subscripts are those defined in this paper. 
  }
  \label{tab:comp}
\end{table}

One of the main reasons to use an HDG method is that all interior
unknowns ($\vu, \phi$) can be eliminated to get a global system for
solely the interface unknowns ($\hat\phi$). This is possible whenever
the local system
\begin{subequations}
  \label{eq:6}
  \begin{align}
    \ii k ( \vu, \vv)_K - (\phi, \dive \vv)_K & = - \ip{\hat \phi,
      \vv\cdot\vn}_{\d K}, && \forall \vv \in V(K),
    \\  \nonumber
    -(\dive \vu, \psi)_K -\ip{\tau \phi, \psi}_{\d K} - \ii k
    (\phi,\psi)_K & = - \ip{ \tau \hat \phi, \psi}_{\d K}
                           \\ 
     & \quad-(f,\psi)_K,&& \forall
    \psi \in W(K),
  \end{align}
\end{subequations}
is uniquely solvable. (For details on this elimination and other
perspectives on HDG methods, see~\cite{CockbGopalLazar09}.) In the lowest order ($p=0$) case, on a
square element $K$ of side length $h$, if we use a basis in the
following order
\[
\vu_1 =
\begin{bmatrix}
  1 \\ 0
\end{bmatrix},\quad \vu_2 =
\begin{bmatrix}
  0 \\ 1
\end{bmatrix},\quad \phi_1 = 1, \qquad \text{ on } K,
\]
then the element matrix for the system~\eqref{eq:6} is
\begin{align*}
  M & =
  \begin{bmatrix}
    \ii{k}\,h^{2}&0&0\\
    &\ii{k}\,h^{2}&0\\
    0&0&-4\,h\tau-\ii{k}\,h^2
  \end{bmatrix}.
\end{align*}
This shows that if
\begin{equation}
  \label{eq:1}
  4\tau =-\ii {k} h,   
\end{equation}
then $M$ is singular, and so the HDG method will fail.  {\em The usual
  recipe of choosing $\tau=1$ is therefore inappropriate when $k$ is
  complex valued.}

\subsection*{Intermediate case of the 2D Maxwell
  system} \label{sec:2dMax}

It is an interesting exercise to consider the 2D Maxwell system before
going to the full 3D case.  In fact, the HDG method for the 2D Maxwell
system can be determined from the HDG method for the 2D Helmholtz
system. The 2D Maxwell system is
\begin{subequations}
  \label{eq:maxwell2d}
  \begin{align}
    \ii k \calE - \scurl \vcalH & = -J,
    \\
    \ii k \vcalH + \curl \calE & = 0,
  \end{align}
\end{subequations}
where $J\in L^2(\Omega)$, and the scalar curl $\scurl \cdot$ and the
vector curl $\curl \cdot$ are defined by
\[
\scurl \vcalH = \d_1 \calH_2 - \d_2 \calH_1 = \dive R (\vcalH), \qquad
\curl \calE = (\d_2 \calE, -\d_1 \calE) = R (\grad \calE).
\]
Here $R (v_1,v_2) = (v_2, -v_1)$ is the operator that rotates vectors
counterclockwise by $+\pi/2$ in the plane. Clearly, if we set $\vcalr
= - R \vcalH$, then~\eqref{eq:maxwell2d} becomes
\begin{align*}
  \ii k \calE + \dive \vcalr & = -J,
  \\
  -\ii k \vcalr + R R \grad \calE & = 0,
\end{align*}
which, since $ R R \vv = -\vv$ (rotation by $\pi$), coincides
with~\eqref{eq:Helmholtz} with $\varPhi = \calE$, $\vcalU = \vcalr$,
and $f=-J$.  This also shows that the HDG method for Helmholtz
equation should yield an HDG method for the 2D Maxwell system.  We
thus conclude that there exist stabilization parameters that will
cause the HDG system for 2D Maxwell system to fail.

To examine this 2D HDG method, if we let $\vH$ and $E$ denote the HDG
approximations for $R\vr$ and $\calE$, respectively, then the
HDG system~\eqref{eq:globalHelmholtz} with $\vu$ and $\phi$ replaced
by $-R\vH$ and $E$, respectively, gives
\begin{align*}
  & \sum_{K\in\Th} & -(E, \curl \vw)_K + \ip{ \hat E, \vn \times
    \vw}_{\d K} - \ii k (\vH, \vw)_K & = 0,
  \\
  & \sum_{K\in\Th} & \ii k ( E, \psi)_K - (\curl \vH, \psi)_K +
  \ip{\tau ( E - \hat E), \;\psi}_{\d K} & = -(\vec
  J,\psi)_{\Omega},
  \\
  & \sum_{K\in\Th} & \ip{ \widehat{R\vH} \cdot \vn, \;\hat \psi }_{\d K} & =
  0,
\end{align*} 
for all $\vw \in R (V_h), \psi \in W_h$ and $\hat \psi \in M_h.$ We
have used the fact that $-(R \vH) \cdot \vn = \vH \cdot \vt$, where
$\vt = R\vn$ the tangent vector, and we have used the 2D cross product
defined by $\vv \times \vn = \vv \cdot \vt$. In particular, the
numerical flux prescription~\eqref{eq:numflux} implies
\[
-\widehat{R\vH} \cdot \vn = -R\vH\cdot \vn + \tau ( E - \hat E),
\]
where $\widehat{R\vH}$ denotes the numerical trace of $R \vH$.  We
rewrite this in terms of $\vH$ and $E$, to obtain
\[
\hat H \cdot \vt = \vH \cdot \vt + \tau ( E - \hat E).
\]
One may rewrite this again, as
\begin{equation}
  \label{eq:5}
  \hat H \times \vn = \vH \times \vn  + 
  \tau ( E - \hat E).
\end{equation}
This expression is notable because it will help us consistently
transition the numerical flux prescription from the Helmholtz to the
full 3D Maxwell case discussed next. A comparison of this formula with
those in the existing literature is included in Table~\ref{tab:comp}.

\subsection*{The 3D Maxwell system}   \label{sec:3dmax}

Consider the 3D Maxwell system on $\Omega\subset\mathbb{R}^3$ with a
perfect electrically conducting boundary condition,
\begin{subequations}
  \label{eq:maxwell}
  \begin{align}
    \ii k \vcalE - \curl \vcalH & = -\vec J, \qquad \text{in } \Omega,
    \\
    \ii k \vcalH + \curl \vcalE & = \vec 0, \qquad \text{in } \Omega, \\
    \vn\times \vcalE &= \vec 0, \qquad \text{on } \partial \Omega,
  \end{align}
\end{subequations}
where $\vec J \in (L^2(\Omega))^3$. For this problem, $\Th$ denotes a
cubic or tetrahedral mesh, and $\mathcal{F}_h$ denotes the collection
of mesh faces. The HDG method approximates the exact solution
$(\vcalE, \vcalH, \hat \calE)$ by the discrete solution $(\vE, \vH,
\hat E)\in Y_h \times Y_h \times N_h$. The discrete spaces are defined
by
\begin{align*}
  Y_h&=\{ \vv\in (L^2(\Omega))^3~:~\vv|_K\in Y(K),~\forall K\in\Th\}\\
  N_h&=\{ \hat\eta\in (L^2(\mathcal{F}_h))^3~:~\hat\eta|_F\in
  N(F),~\forall F\in\mathcal{F}_h \text{ and }
  \hat\eta|_{\partial\Omega}=\vec 0\},
\end{align*}
with polynomial spaces $Y(K)$ and $N(F)$ specified by:
\begin{align*}
  &\hspace{-0.7cm}\text{\underline{Tetrahedra}}\qquad\qquad&&\hspace{-0.4cm}\text{\underline{Cubes}}
  \\
  Y(K)&=(\mathcal{P}_p(K))^3 &Y(K)&=(\mathcal{Q}_p(K))^3
  \\
  N(F)&=\{ \hat\eta \in (\mathcal{P}_p(F))^3~:~ \hat\eta\cdot\vn=0 \}
  &N(F)&=\{ \hat\eta \in (\mathcal{Q}_p(F))^3~:~ \hat\eta\cdot\vn=0 \}
  \\
\end{align*}

Our HDG method for~\eqref{eq:maxwell} is
\begin{align*}
  & \sum_{K\in\Th} & \ii k ( \vE, \vv)_K - (\curl \vH, \vv)_K +
  \ip{(\hat H - H) \times \vn, \;\vv}_{\d K} & = -(\vec
  J,\vv)_{\Omega}, && \forall \vv \in Y_{h},
  \\
  & \sum_{K\in\Th} & -(\vE, \curl \vw)_K + \ip{ \hat E, \vn \times
    \vw}_{\d K} - \ii k (\vH, \vw)_K & = 0, && \forall \vw \in Y_{h},
  \\
  & \sum_{K\in\Th} & \ip{ \hat H \times \vn, \;\hat w }_{\d K} & = 0,
  && \forall \hat w \in N_{h},
\end{align*}
where, in analogy with~\eqref{eq:5}, we now set numerical flux by
\begin{equation}
  \label{eq:4}
\hat H \times \vn = \vH \times \vn + \tau ( \vE - \hat E)_t,  
\end{equation}
where $ ( \vE - \hat E)_t$ denotes the tangential component, or
equivalently
\[
\hat H \times \vn = \vH \times \vn + \tau (\vn \times ( \vE - \hat E))
\times \vn.
\]
Note that the 2D system~\eqref{eq:maxwell2d} is obtained from the 3D
Maxwell system~\eqref{eq:maxwell} by assuming symmetry in
$x_3$-direction. Hence, for consistency between 2D and 3D
formulations, we should have the same form for the numerical flux
prescriptions in 2D and 3D.

The HDG method is then equivalently written as
\begin{subequations}
  \label{eq:hdgmaxwell}
  \begin{align}
    & \sum_{K\in\Th} & \ii k ( \vE, \vv)_K - (\curl \vH, \vv)_K +
    \ip{\tau (\vE - \hat E) \times \vn, \;\vv\times \vn}_{\d K} & =
    -(\vec J,\vv)_{\Omega}, 
    \label{eq:hdgmaxwella}\\
    & \sum_{K\in\Th} & -(\vE, \curl \vw)_K + \ip{ \hat E, \vn \times
      \vw}_{\d K} - \ii k (\vH, \vw)_K & = 0,
    \label{eq:hdgmaxwellb}\\
    & \sum_{K\in\Th} & \ip{ \vH + \tau \,\vn \times ( \vE - \hat E) ,
      \;\hat w \times \vn }_{\d K} & = 0, \label{eq:hdgmaxwellc}
  \end{align}
\end{subequations}
for all $\vv, \vw \in Y_{h},$ and
$\hat w \in N_{h}$.  For comparison with other existing formulations,
see Table~\ref{tab:comp}.

Again, let us look at the solvability of the {\em local element
  problem}
\begin{subequations}
  \label{eq:localmaxwell}
  \begin{align}
    \ii k ( \vE, \vv)_K - (\curl \vH, \vv)_K + \ip{\tau \vE \times
      \vn, \;\vv\times \vn}_{\d K} & = \ip{\tau \hat E \times \vn,
      \;\vv\times \vn}_{\d K}   \nonumber                                      
   \\ & -(\vec J,\vv)_K,
    \label{eq:localmaxwella}\\
    -(\vE, \curl \vw)_K - \ii k (\vH, \vw)_K & = - \ip{ \hat E, \vn
      \times \vw}_{\d K}, 
    \label{eq:localmaxwellb}
  \end{align}
\end{subequations}
for all $\vv, \vw \in Y(K)$.
In the lowest order ($p=0$) case, on a cube element $K$ of side length
$h$, if we use a basis in the following order
\begin{equation}\label{eq:basis}
  \vE_1 =
  \begin{bmatrix}
    1 \\ 0 \\ 0
  \end{bmatrix},\quad \vE_2 =
  \begin{bmatrix}
    0 \\ 1 \\ 0
  \end{bmatrix},\quad \vE_3 =
  \begin{bmatrix}
    0 \\ 0 \\ 1
  \end{bmatrix}, \quad \vH_1 =
  \begin{bmatrix}
    1 \\ 0 \\ 0
  \end{bmatrix},\quad \vH_2 =
  \begin{bmatrix}
    0 \\ 1 \\ 0
  \end{bmatrix},\quad \vH_3 =
  \begin{bmatrix}
    0 \\ 0 \\ 1
  \end{bmatrix},
\end{equation}
then the $6 \times 6$ element matrix for the
system~\eqref{eq:localmaxwell} is
\[
M =
\begin{bmatrix}
  (4 h^2 \tau + \ii k h^3)I_3 & 0
  \\
  0 & -(\ii k h^3) I_3
\end{bmatrix},
\]
where $I_3$ denotes the $3\times 3$ identity matrix.
Again, exactly as in the Helmholtz case -- cf.~\eqref{eq:1} -- we find
that if
\begin{equation}
  \label{eq:2}
  4 \tau = -\ii k h,
\end{equation}
then the local static condensation required in the HDG method will
fail in the Maxwell case also.

\subsection*{Behavior on tetrahedral meshes}   \label{sec:tetr}

For the lowest order ($p=0$) case on a tetrahedral element, just as
for the cube element described above, there are bad stabilization
parameter values. Consider, for example, the tetrahedral element of
size $h$ defined by
\begin{equation}
  \label{eq:3}
K=\{\vec x \in \mathbb{R}^3 ~:~ x_j \ge 0~ \forall j,~ x_1+x_2+x_3\leq
h\},
\end{equation}
with a basis ordered as in \eqref{eq:basis}. The element matrix for
the system \eqref{eq:localmaxwell} is then
\[
M=\frac{1}{6}
\begin{bmatrix}
  (2\sqrt{3}+6)h^2\tau+\ii kh^3\!\!\!\!\!& -\sqrt{3}h^2\tau & -\sqrt{3}h^2\tau & 0 & 0 & 0 \\
  -\sqrt{3}h^2\tau & (2\sqrt{3}+6)h^2\tau+\ii kh^3\!\!\!\!& -\sqrt{3}h^2\tau & 0 & 0 & 0 \\
  -\sqrt{3}h^2\tau & -\sqrt{3}h^2\tau & 4 h^2 \tau + \ii k h^3 & 0 & 0 & 0\\
  0 & 0 & 0 & -\ii k h^3 & 0 & 0 \\
  0 & 0 & 0 & 0 & -\ii k h^3 & 0 \\
  0 & 0 & 0 & 0 & 0 & -\ii k h^3
\end{bmatrix}.
\]
We immediately see that the rows become linearly dependent if 
\[
(3\sqrt{3}+6)\tau=-\ii kh.
\]
Hence, for $\tau=-\ii kh / (3\sqrt{3}+6)$, the HDG method
will fail on tetrahedral meshes.

\begin{figure}
  \centering
  \begin{subfigure}[b]{0.40\textwidth}
    \includegraphics[width=\textwidth]{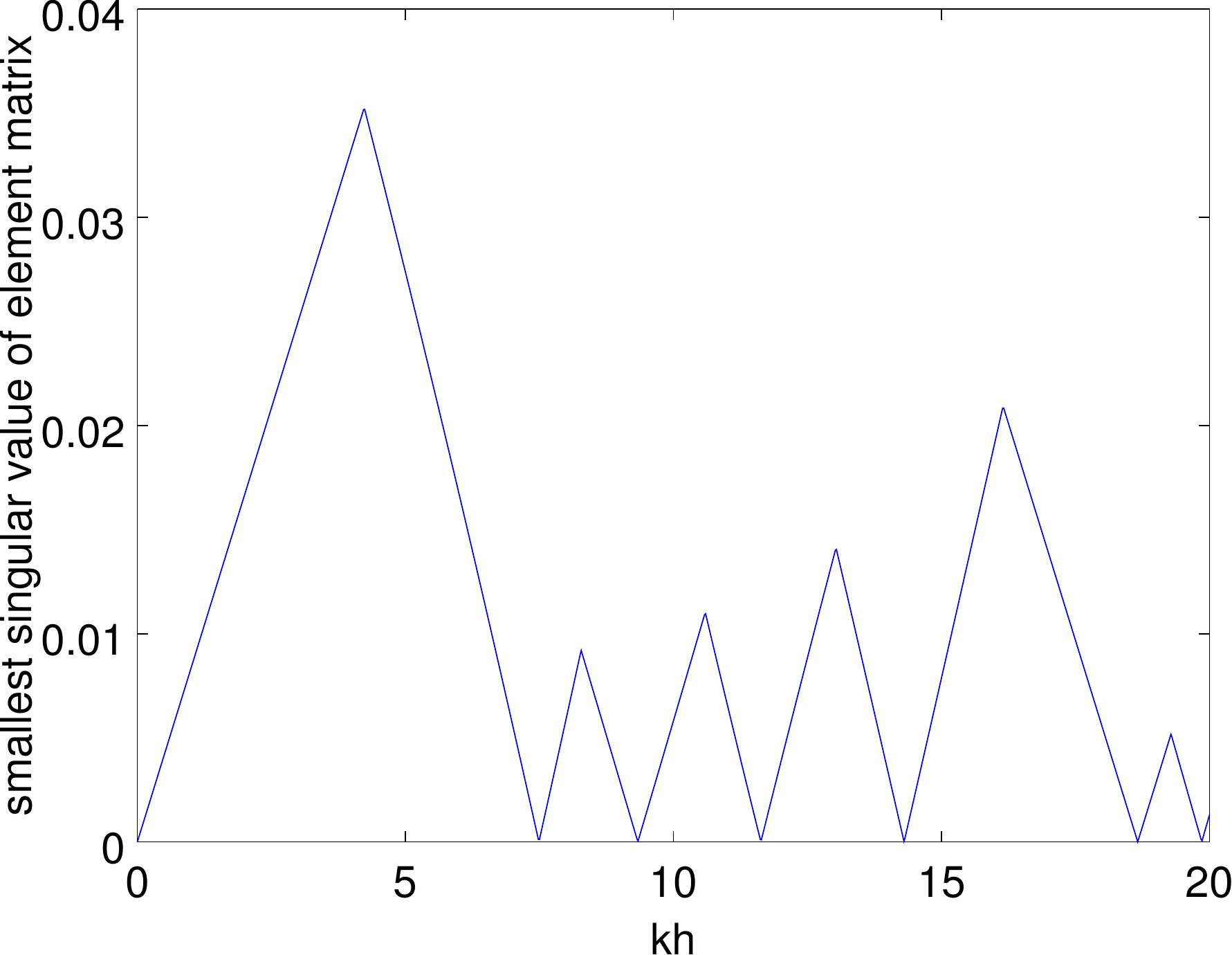}
    \caption{$p=1$, $\tau=-\ii$}
    \label{fig:p1tau-i}
  \end{subfigure}
  \qquad
  \begin{subfigure}[b]{0.40\textwidth}
    \includegraphics[width=\textwidth]{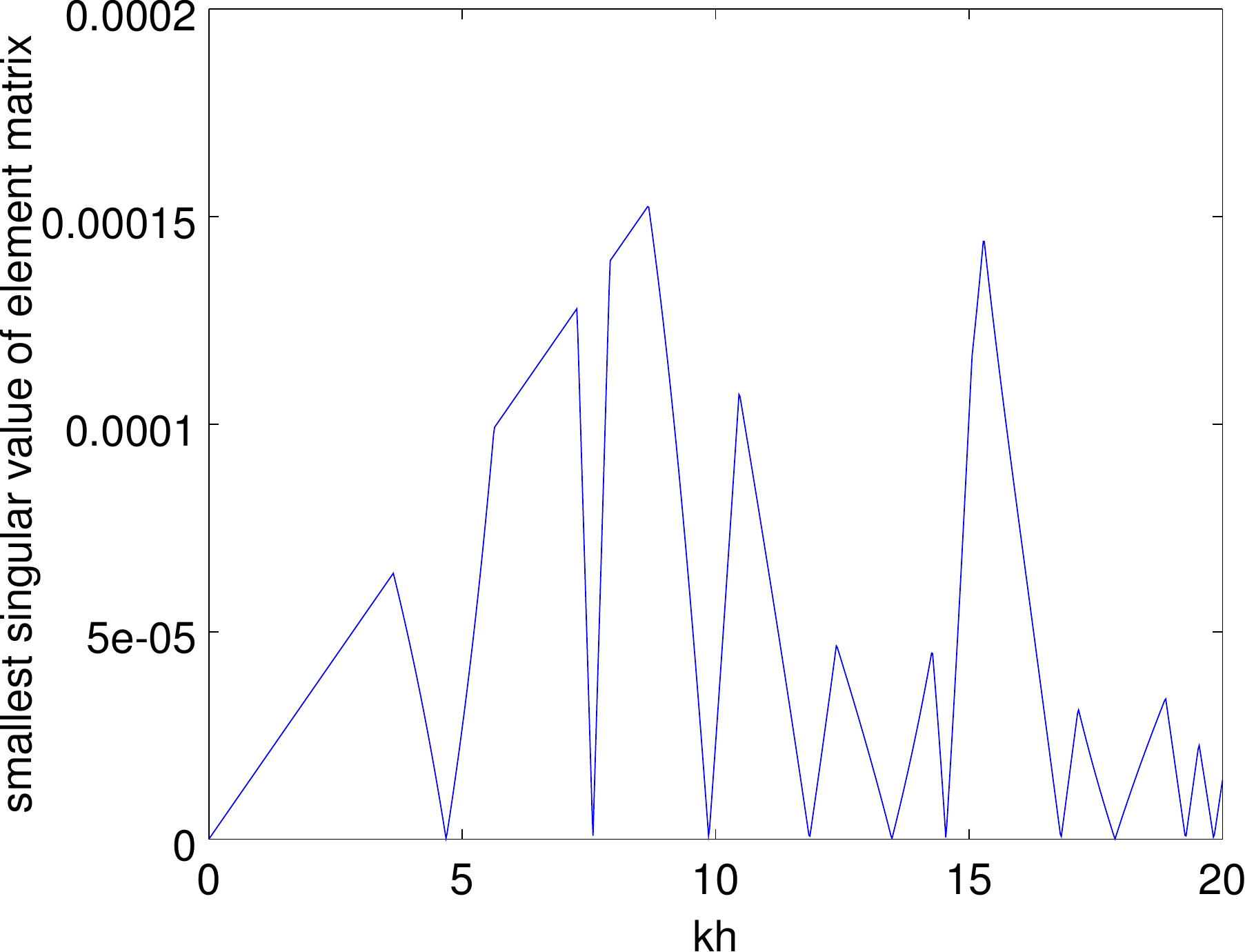}
    \caption{$p=2$, $\tau=-\ii$}
    \label{fig:p2tau-i}
  \end{subfigure}
  \begin{subfigure}[b]{0.40\textwidth}
    \includegraphics[width=\textwidth]{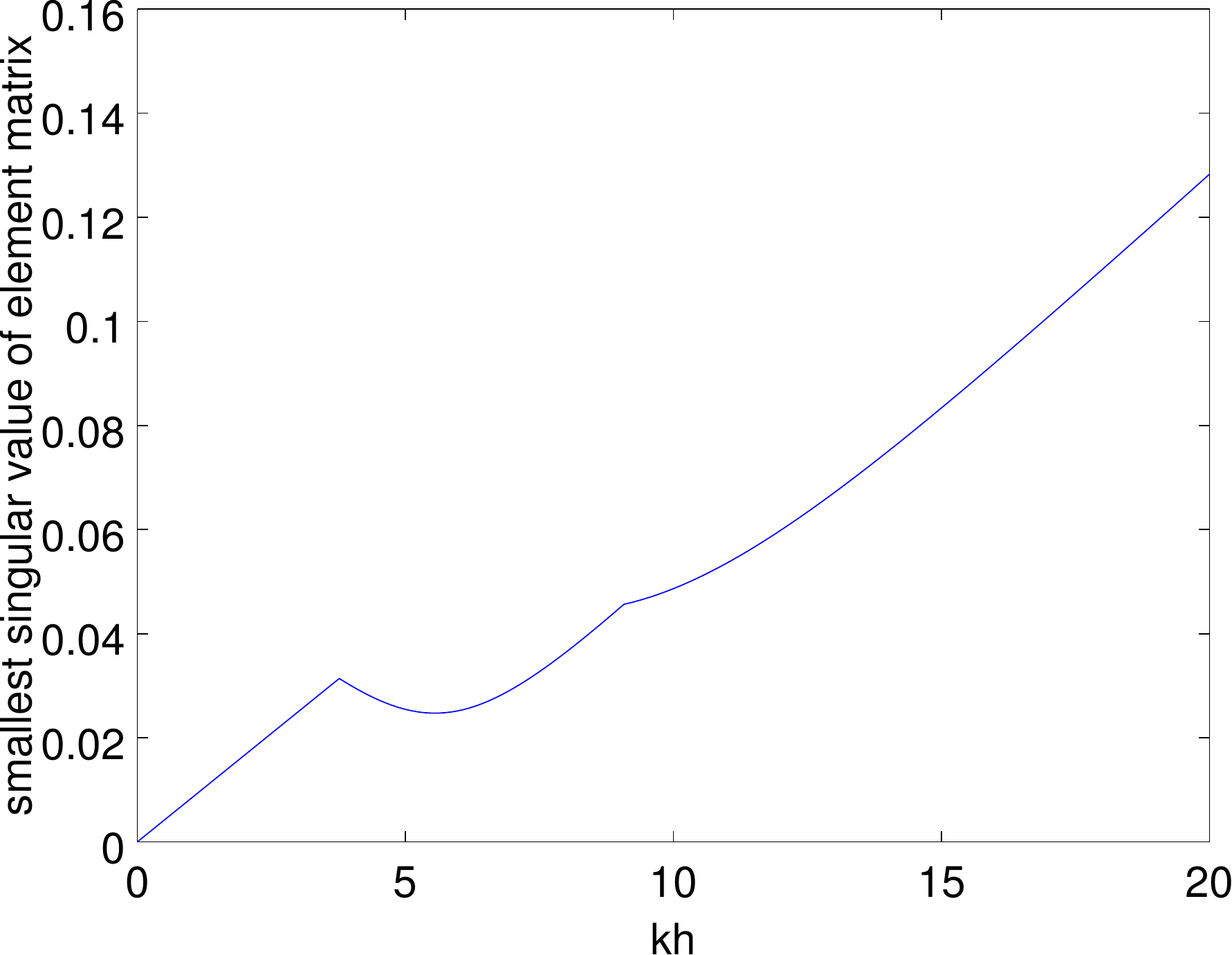}
    \caption{$p=1$, $\tau=1$}
    \label{fig:p1tau1}
  \end{subfigure}
  \qquad
  \begin{subfigure}[b]{0.40\textwidth}
    \includegraphics[width=\textwidth]{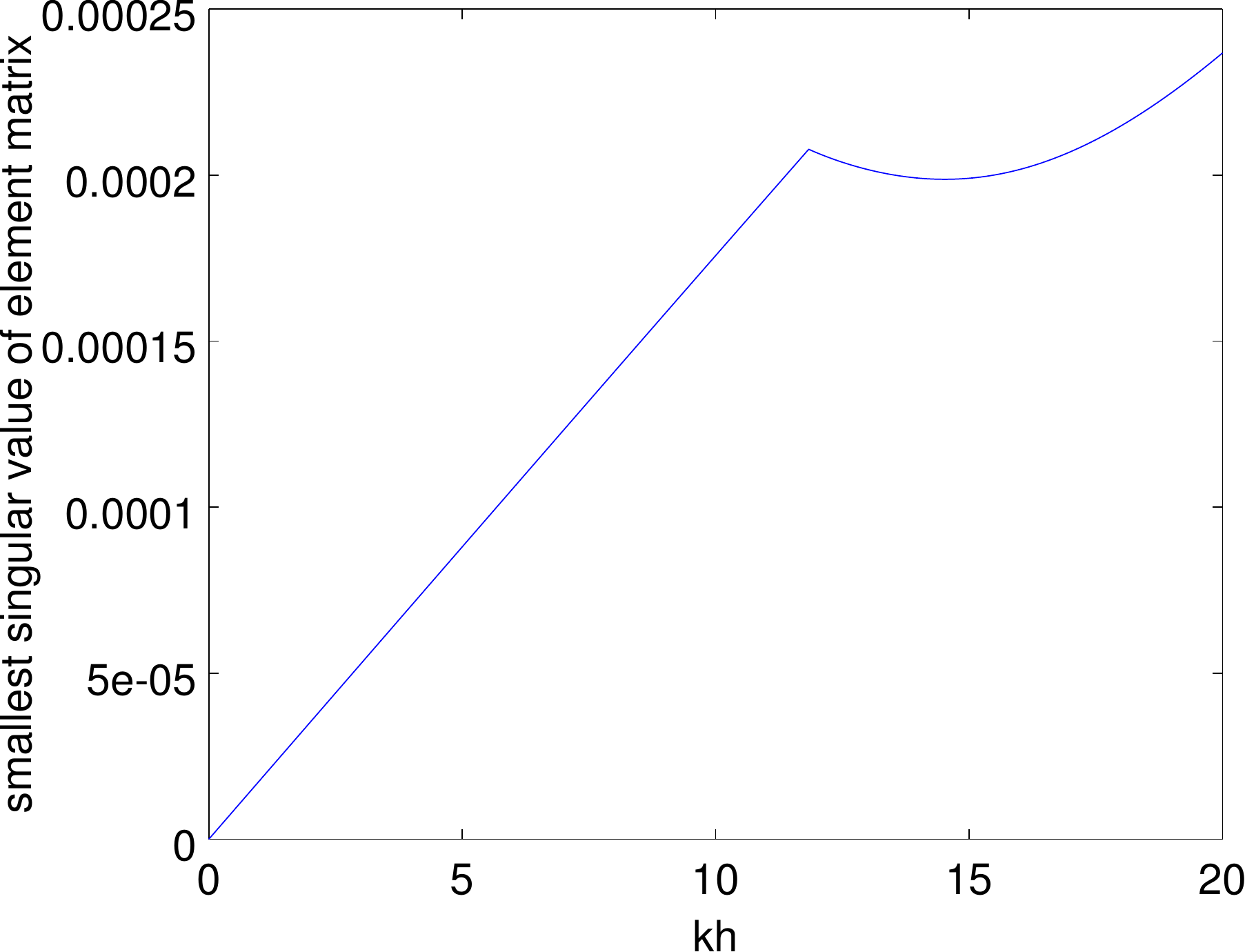}
    \caption{$p=2$, $\tau=1$}
    \label{fig:p2tau1}
  \end{subfigure}
  \centering
  \begin{subfigure}[b]{0.45\textwidth}
    \includegraphics[width=\textwidth]{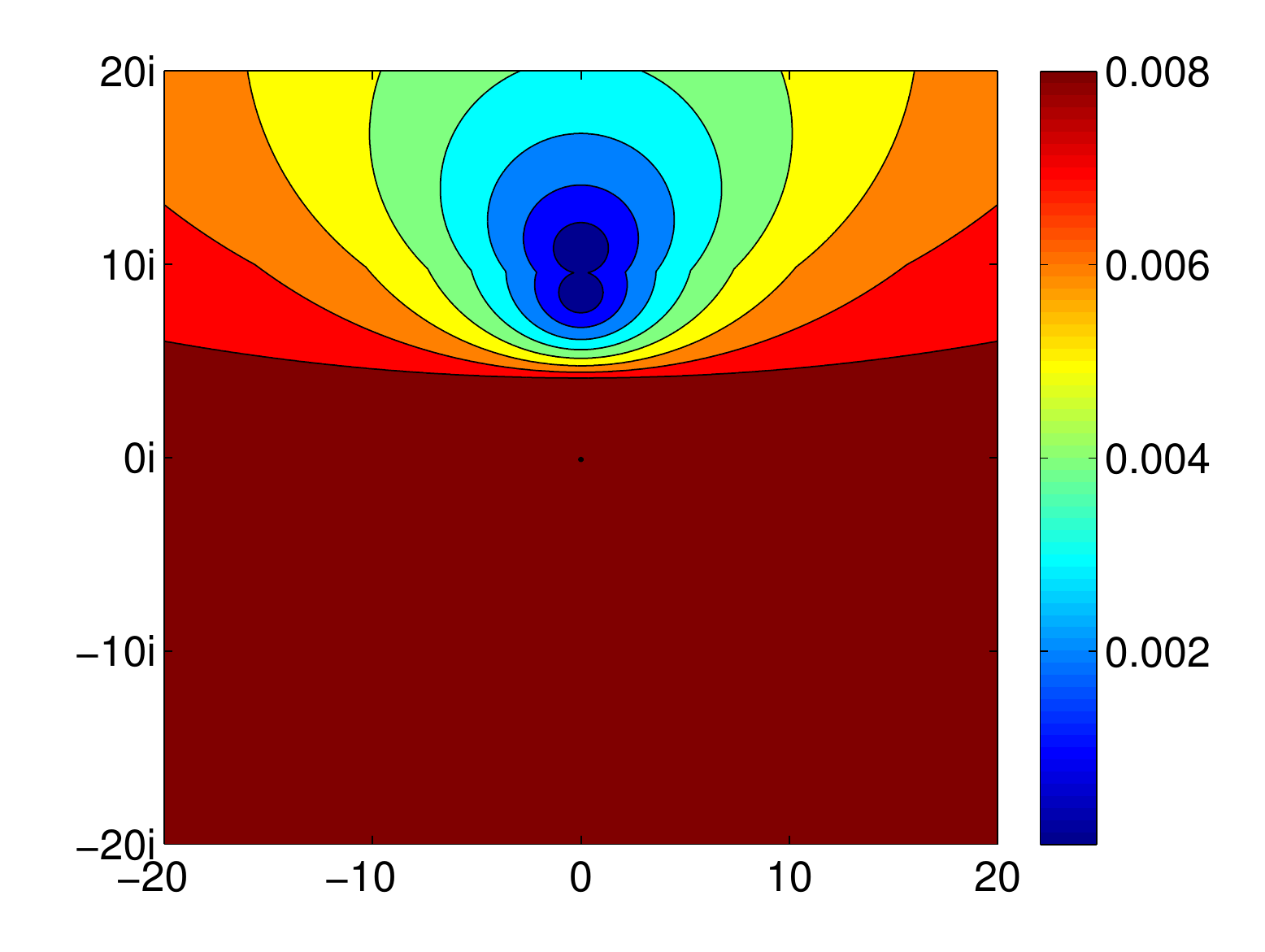}
    \caption{$kh=1$, $p=1$} 
    \label{fig:plane_realkh_1}
  \end{subfigure}
  \quad
  \begin{subfigure}[b]{0.45\textwidth}
    \includegraphics[width=\textwidth]{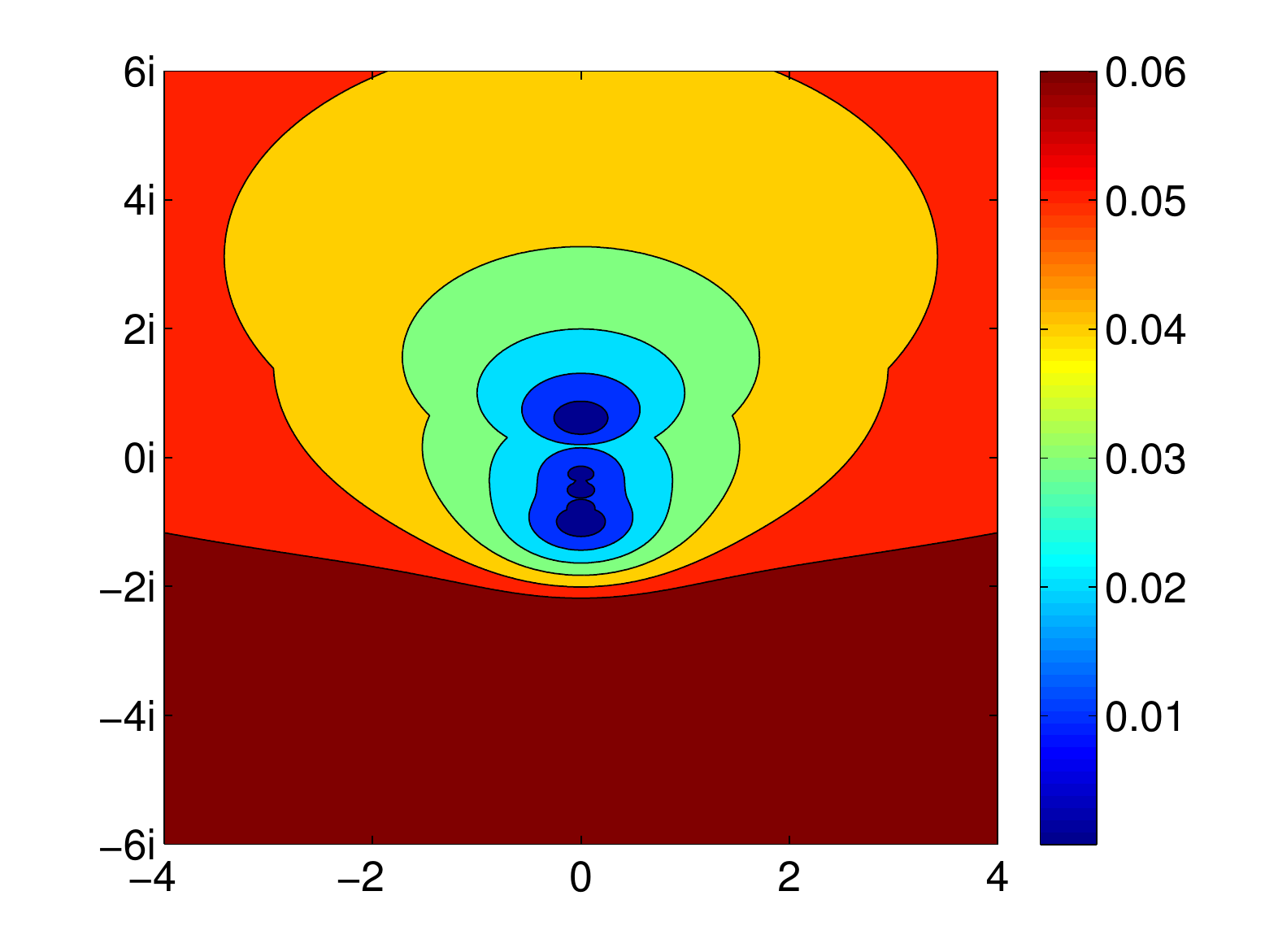}
    \caption{$kh\approx 7.49$, $p=1$}
    \label{fig:plane_realkh_2}
  \end{subfigure}
  \begin{subfigure}[b]{0.45\textwidth}
    \includegraphics[width=\textwidth]{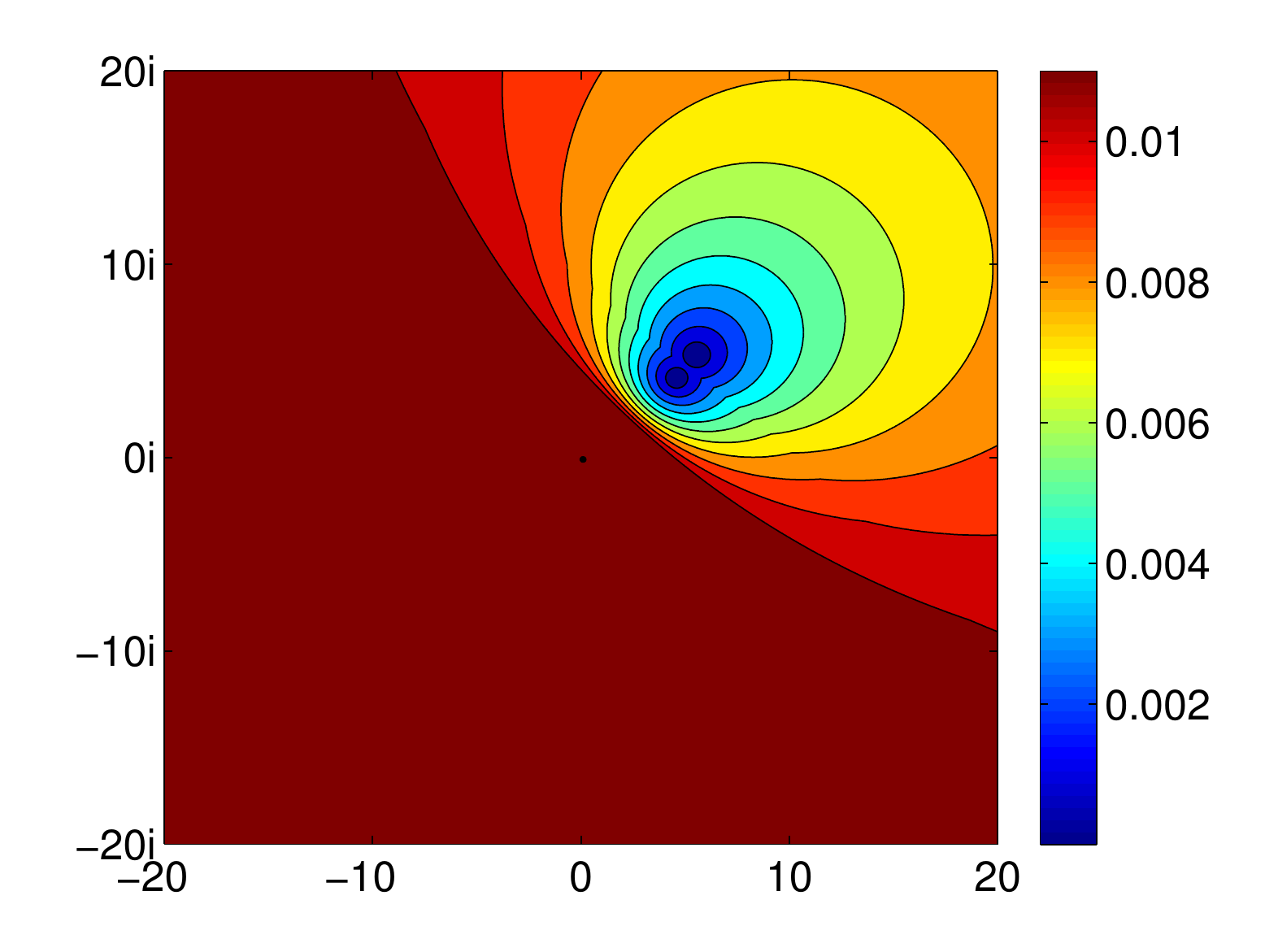}
    \caption{$kh=1+\ii$, $p=1$}
    \label{fig:plane_1plusi_1}   
  \end{subfigure}
  \quad
  \begin{subfigure}[b]{0.45\textwidth} \includegraphics[width=\textwidth]{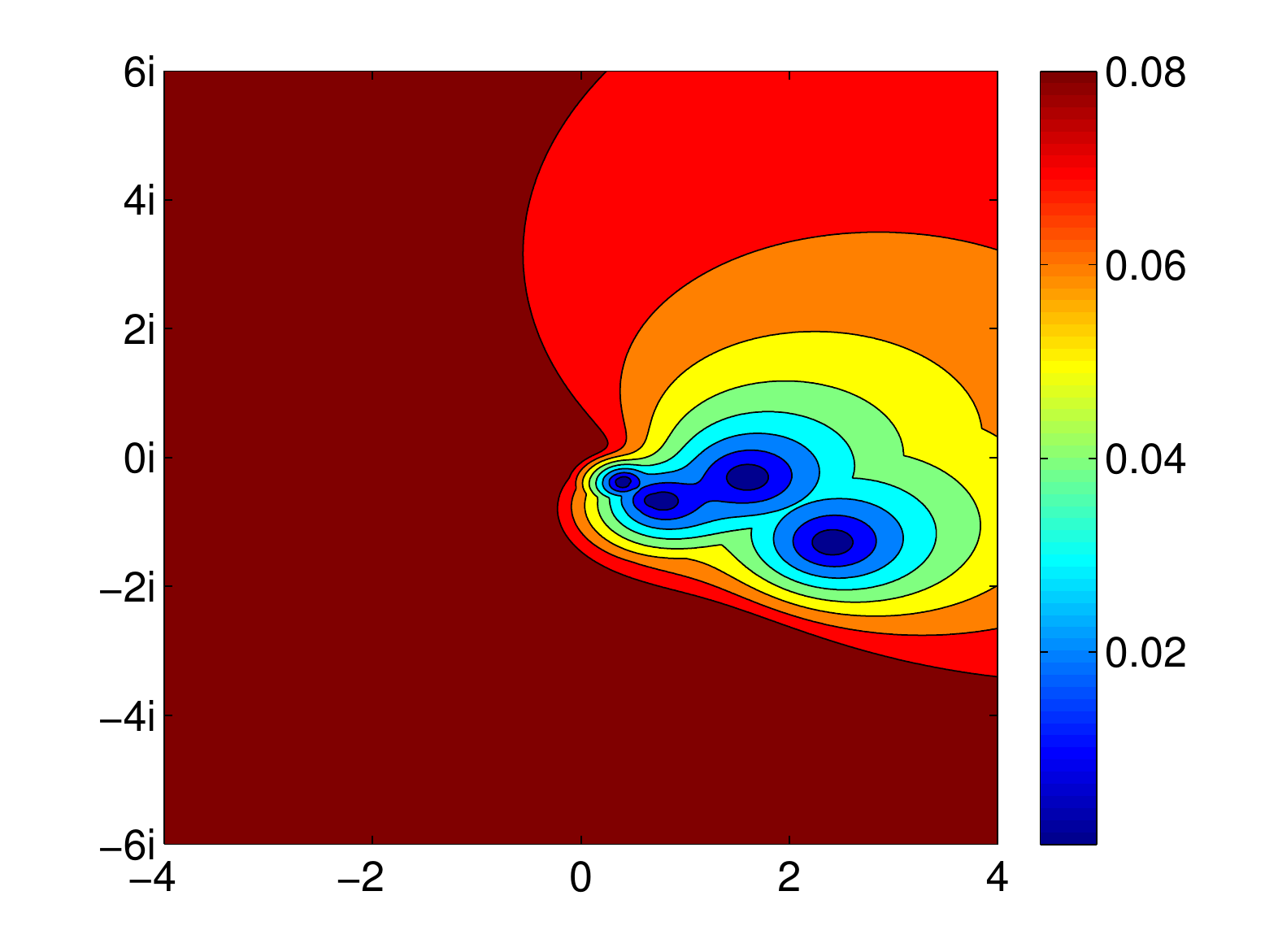}
    \caption{$kh\approx 7.49(1+\ii)$, $p=1$}
    \label{fig:plane_1plusi_2}
  \end{subfigure}
  \caption{The smallest singular values of a tetrahedral HDG element matrix}
  \label{fig:plane}
\end{figure}

For orders $p \geq 1$, the element matrices are too complex to find
bad parameter values so simply. Instead, we experiment
numerically. Setting $\tau=-\ii$, which is equivalent to the choice
made in~\cite{NguyePeraiCockb11a} (see Table~\ref{tab:comp}), we
compute the smallest singular value of the element matrix~$M$ (the
matrix of the left hand side of~\eqref{eq:localmaxwell} with $K$ set
by~\eqref{eq:3}) for a range of normalized wavenumbers
$kh$. Figures~\ref{fig:p1tau-i} and~\ref{fig:p2tau-i} show that, for
orders $p=1$ and $p=2$, there are values of $kh$ for which $\tau=-\ii$
results in a singular value very close to zero.  Taking a closer look
at the first nonzero local minimum in Figure~\ref{fig:p1tau-i}, we
find that the local matrix corresponding to normalized wavenumber
$kh \approx 7.49$ has an estimated condition number exceeding
$3.9\times 10^{15}$, i.e., for all practical purposes, the element matrix is
singular. To illustrate how a different choice of stabilization
parameter $\tau$ can affect the conditioning of the element matrix,
Figures~\ref{fig:p1tau1} and~\ref{fig:p2tau1} show the smallest
singular values for the same range of $kh$, but with $\tau=1$. Clearly
the latter choice of $\tau$ is better than the former.

From another perspective, Figure~\ref{fig:plane_realkh_1} shows the
smallest singular value of the element matrix as $\tau$ is varied in
the complex plane, while fixing $kh$ to $1$.
Figure~\ref{fig:plane_realkh_2} is similar except that we fixed $kh$
to the value discussed above, approximately $7.49$.  In both cases, we
find that the worst values of $\tau$ are along the imaginary axis.
Finally, in Figures~\ref{fig:plane_1plusi_1}
and~\ref{fig:plane_1plusi_2}, we see the effects of multiplying these
real values of $kh$ by $1+\ii$. The region of the complex plane where
bad values of $\tau$ are found changes significantly when $kh$ is
complex.

\section*{Results on unisolvent stabilization}

We now turn to the question of how we can choose a value for the
stabilization parameter $\tau$ that will guarantee that the local
matrices are not singular.  The answer, given by a condition on
$\tau$, surprisingly also guarantees that the global condensed HDG
matrix is nonsingular. These results are based on a tenuous stability
inherited from the fact nonzero polynomials are never waves, stated
precisely in the ensuing lemma. Then we give the condition on $\tau$
that guarantees unisolvency, and before concluding the section,
present some caveats on relying solely on this tenuous stability.

As is standard in all HDG methods, the unique solvability of the
element problem allows the formulation of a condensed global problem
that involves only the interface unknowns. We introduce the following
notation to describe the condensed systems. First, for Maxwell's
equations, for any $\eta\in N_h$, let $(\vE^{\eta},\vH^{\eta})\in
Y_h\times Y_h$ denote the fields such that, for each $K\in\Th$, the
pair $(\vE^{\eta}|_K,\vH^{\eta}|_K)$ satisfies the local
problem~\eqref{eq:localmaxwell} with data $\eta|_{\partial K}$. That
is,
\begin{subequations}
  \label{eq:rhseta}
  \begin{align}
    \ii k ( \vE^{\eta}, \vv)_K - (\curl \vH^{\eta}, \vv)_K + \ip{\tau
      \vE^{\eta} \times \vn, \;\vv\times \vn}_{\d K} & = \ip{\tau \eta
      \times \vn, \;\vv\times \vn}_{\d K},
    \label{eq:rhsetaa}\\
    -(\vE^{\eta}, \curl \vw)_K - \ii k (\vH^{\eta}, \vw)_K & = - \ip{
      \eta, \vn \times \vw}_{\d K},
    \label{eq:rhsetab}
  \end{align}
\end{subequations}
for all $\vv \in Y(K),\; \vw \in Y(K).$
If all the sources in~\eqref{eq:hdgmaxwell} vanish, then the {\em
  condensed global problem} for $\hat E \in N_h$ takes the form
\begin{equation}
  \label{eq:reduced} 
  a(\hat E,\eta)=0,\quad\forall\eta\in N_h,
\end{equation}
where
\[
a(\Lambda,\eta)=\sum_{K\in\Th} \langle\hat H^{\Lambda}\times\vn,\eta
\rangle_{\partial K}.
\]
By following a standard procedure~\cite{CockbGopalLazar09} we can
express $a(\cdot,\cdot)$ explicitly as follows:
\begin{align*}
  a(\Lambda,\eta)&=\sum_{K\in\Th}\langle\vH^{\Lambda}\times\vn,\eta\rangle_{\partial K}+\langle(\hat H^{\Lambda}-\vH^{\Lambda})\times\vn,\eta\rangle_{\partial K}\\
  &=\sum_{K\in\Th}\ii\overline{k}(\vH^{\Lambda},\vH^{\eta})_{K}-(\curl\vH^{\Lambda},\vE^{\eta})_{K}+\langle\tau\vn\times(\vn\times(\Lambda-\vE^{\Lambda})),\eta\rangle_{\partial K}\\
  &=\sum_{K\in\Th}\ii\overline{k}(\vH^{\Lambda},\vH^{\eta})_{K}-\ii k(\vE^{\Lambda},\vE^{\eta})_{K}+\langle\tau\vn\times(\Lambda-\vE^{\Lambda}),\vn\times\vE^{\eta}\rangle_{\partial K}\\
  &\qquad\quad-\langle\tau\vn\times(\Lambda-\vE^{\Lambda}),\vn\times\eta\rangle_{\partial K}\\
  &=\sum_{K\in\Th}\ii\overline{k}(\vH^{\Lambda},\vH^{\eta})_{K}-\ii
  k(\vE^{\Lambda},\vE^{\eta})_{K}-\tau\langle\vn\times(\Lambda-\vE^{\Lambda}),\vn\times(\eta-\vE^{\eta})\rangle_{\partial
    K}.
\end{align*}
Here we have used the complex conjugate of~\eqref{eq:rhsetab} with
$\vw=\vH^{\Lambda}$, along with the definition of $\hat H^{\Lambda}$,
and then used~\eqref{eq:rhsetaa}. 

Similarly, for the Helmholtz equation, let $(\vu^{\eta},
\phi^{\eta})\in V_h\times W_h$ denote the fields such that, for all
$K\in\Th$, the functions $(\vu^{\eta}|_K, \phi^{\eta}|_K)$ solve the
element problem~\eqref{eq:6} for given data $\hat \phi = \eta$. If the
sources in~\eqref{eq:globalHelmholtz} vanish, then the {\em condensed
  global problem} for $\hat \phi \in M_h$ is written as
\begin{equation}\label{eq:reducedHelmholtz}
  b(\hat \phi, \eta)=0,\quad\forall\eta\in M_h, 
\end{equation}
where the form is found, as before, by the standard procedure:
\begin{align*}
  b(\Lambda,\eta)&= \sum_{K\in\Th}\langle \hat u^{\Lambda}\cdot\vn,\eta \rangle_{\d K}\\
  &=\sum_{K\in\Th}\ii\overline{k}(\vu^{\Lambda},\vu^{\eta})_K-\ii
  k(\phi^{\Lambda},\phi^{\eta})_K-\tau\langle \Lambda-\phi^{\Lambda},
  \eta-\phi^{\eta} \rangle_{\d K}.
\end{align*}
The sesquilinear forms $a(\cdot,\cdot)$ and $b(\cdot, \cdot)$ are used
in the main result, which gives sufficient conditions for the
solvability of the local
problems~\eqref{eq:localmaxwell},~\eqref{eq:6} and the global
problems~\eqref{eq:reduced},~\eqref{eq:reducedHelmholtz}.

Before proceeding to the main result, we give a simple lemma,
which roughly speaking, says that nontrivial harmonic {\em waves are
  not polynomials}.

\begin{lemma}\label{lem:trivial}
  Let $p\geq 0$ be an integer, $0\neq k\in\mathbb{C}$, and $D$ an open
  set.  Then, there is no nontrivial $\vE\in (\mathcal{P}_p(D))^3$
  satisfying
  \[
  \curl( \curl \vE ) - k^2 \vE = 0
  \]
  and there is no nontrivial $\phi \in \mathcal P_p(D)$ satisfying
  \[
  \Delta \phi + k^2 \phi =0.
  \]
\end{lemma}
\begin{proof}
  We use a contradiction argument.  If $E \not\equiv \vec 0$, then we
  may assume without loss of generality that at least one of the
  components of $\vE$ is a polynomial of degree exactly~$p$. But this
  contradicts $k^2 \vE= \curl( \curl \vE )$ because all components of
  $\curl(\curl \vE)$ are polynomials of degree at most $p-2$. Hence
  $\vE \equiv \vec 0$. An analogous argument can be used for the
  Helmholtz case as well.
\end{proof}

\begin{theorem}\label{thm:1}
  Suppose
  \begin{subequations}\label{eq:conditions}
    \begin{align}
      \re (\tau) & \ne 0,   && \text{ whenever } \im (k) =0, \text{ and }\label{eq:conditiona}\\
      \im (k) \re (\tau) & \le 0, && \text{ whenever } \im (k) \ne
      0.\label{eq:conditionb}
    \end{align}
  \end{subequations}
  Then, in the Maxwell case, the local element
  problem~\eqref{eq:localmaxwell} and the condensed global
  problem~\eqref{eq:reduced} are both unisolvent.  Under the same
  condition, in the Helmholtz case, the local element
  problem~\eqref{eq:6} and the condensed global
  problem~\eqref{eq:reducedHelmholtz} are also unisolvent.
\end{theorem}
\begin{proof}
  We first prove the theorem for the local problem for Maxwell's
  equations.  Assume~\eqref{eq:conditions} holds and set $\hat E =
  \vec 0$ in the local problem~\eqref{eq:localmaxwell}. Unisolvency
  will follow by showing that $\vE$ and $\vH$ must equal $\vec
  0$. Choosing $\vv=\vE$, and $\vw=\vH$, then
  subtracting~\eqref{eq:localmaxwellb} from~\eqref{eq:localmaxwella},
  we get
  \[
  \ii k \left ( ||\vE||_K^2 + ||\vH||_K^2 \right ) + 2\ii \mathrm{Im}
  (\vE, \curl \vH )_K + \tau ||\vE \times \vn ||_{\partial K}^2 = 0,
  \]
  whose real part is
  \[
  -\mathrm{Im}(k) \left ( ||\vE||_K^2 + ||\vH||_K^2 \right ) +
  \mathrm{Re}(\tau) ||\vE \times \vn ||_{\partial K}^2 = 0.
  \]
  Under condition~\eqref{eq:conditionb}, we immediately have that the
  fields $\vE$ and $\vH$ are zero on
  $K$. Otherwise,~\eqref{eq:conditiona} implies $\vE \times
  \vn|_{\partial K} = 0$, and then~\eqref{eq:localmaxwell} gives
  \begin{align*}
    \ii k \vE - \curl \vH & = 0,
    \\
    \ii k \vH + \curl \vE & = 0,
  \end{align*}
  implying
  \[
  \curl( \curl \vE ) = k^2 \vE.
  \]
  By Lemma~\ref{lem:trivial} this equation has no nontrivial solutions
  in the space $Y(K)$. Thus, the element problem for Maxwell's
  equations is unisolvent.
 
  We prove that the global problem for Maxwell's equations is
  unisolvent by showing that $\hat E=\vec 0$ is the unique solution of
  equation~\eqref{eq:reduced}. This is done in a manner almost
  identical to what was done above for the local problem: 
  First, set $\eta=\hat E$ in equation~\eqref{eq:reduced} and take the
  real part to get 
  \begin{equation}\label{eq:realpart}
    \sum_{K\in\Th} \text{Im}(k)\left (||\vH||^2_{K}+||\vE||^2_{K}\right
    )-\textrm{Re}(\tau)||\vn\times(\hat E-\vE)||^2_{\partial K}=0.
  \end{equation}
  This immediately shows that if condition~\eqref{eq:conditionb}
  holds, then the fields $\vE$ and $\vH$ are zero on
  $\Omega\subset\mathbb{R}^3$ and the proof is finished.  In the case
  of condition~\eqref{eq:conditiona}, we have $\vn\times(\hat
  E-\vE|_{\partial K})=\vec 0$ for all $K$. Using
  equations~\eqref{eq:hdgmaxwell}, this yields
  \[
  [\curl (\curl\vE)]|_K=k^2\vE|_K,
  \]
  so Lemma~\ref{lem:trivial} proves that the fields on element
  interiors are zero, which in turn implies $\hat E=\vec 0$
  also. Thus, the theorem holds for the Maxwell case.

  The proof for the Helmholtz case is entirely analogous.
\end{proof}

Note that even with Dirichlet boundary conditions and real $k$, the
theorem asserts the existence of a unique solution for the Helmholtz
equation. However, the exact Helmholtz problem~\eqref{eq:Helmholtz} is
well-known to be {\em not} uniquely solvable when $k$ is set to one of
an infinite sequence of real resonance values. The fact that the
discrete system is uniquely solvable even when the exact system is
not, suggests the presence of artificial dissipation in HDG
methods. We will investigate this issue more thoroughly in the next
section. 

However, we do not advocate relying on this discrete unisolvency near
a resonance where the original boundary value problem is not uniquely
solvable. The discrete matrix, although invertible, can be poorly
conditioned near these resonances. Consider, for example, the
Helmholtz equation on the unit square with Dirichlet boundary
conditions. The first resonance occurs at $k=\pi\sqrt{2}$. In
Figure~\ref{fig:conditioning}, we plot the condition number
$\sigma_{\rm max}/\sigma_{\rm min}$ of the condensed HDG matrix for a
range of wavenumbers near the resonance $k=\pi\sqrt{2}$, using a small
fixed mesh of mesh size $h=1/4$, and a value of $\tau=1$ that
satisfies~\eqref{eq:conditions}. We observe that although the
condition number remains finite, as predicted by the theorem, it peaks
near the resonance for both the $p=0$ and the $p=1$ cases.  We also
observe that a parameter setting of $\tau=-\ii$ that does not satisfy
the conditions of the theorem produce much larger condition numbers,
e.g., the condition numbers that are orders of magnitude greater than
$10^{10}$ (off axis limits of Figure~\ref{fig:conditioningp1}) for $k$
near the resonance were obtained for $p=1$ and $\tau=-\ii$. To
summarize the caveat, {\em even though the
  condition number is always bounded for values of~$\tau$ that
  satisfy~\eqref{eq:conditions}, it may still be practically
  infeasible to find the HDG solution near a resonance.}

\begin{figure}
  \centering
  \begin{subfigure}[b]{0.45\textwidth}
    \includegraphics[width=\textwidth]{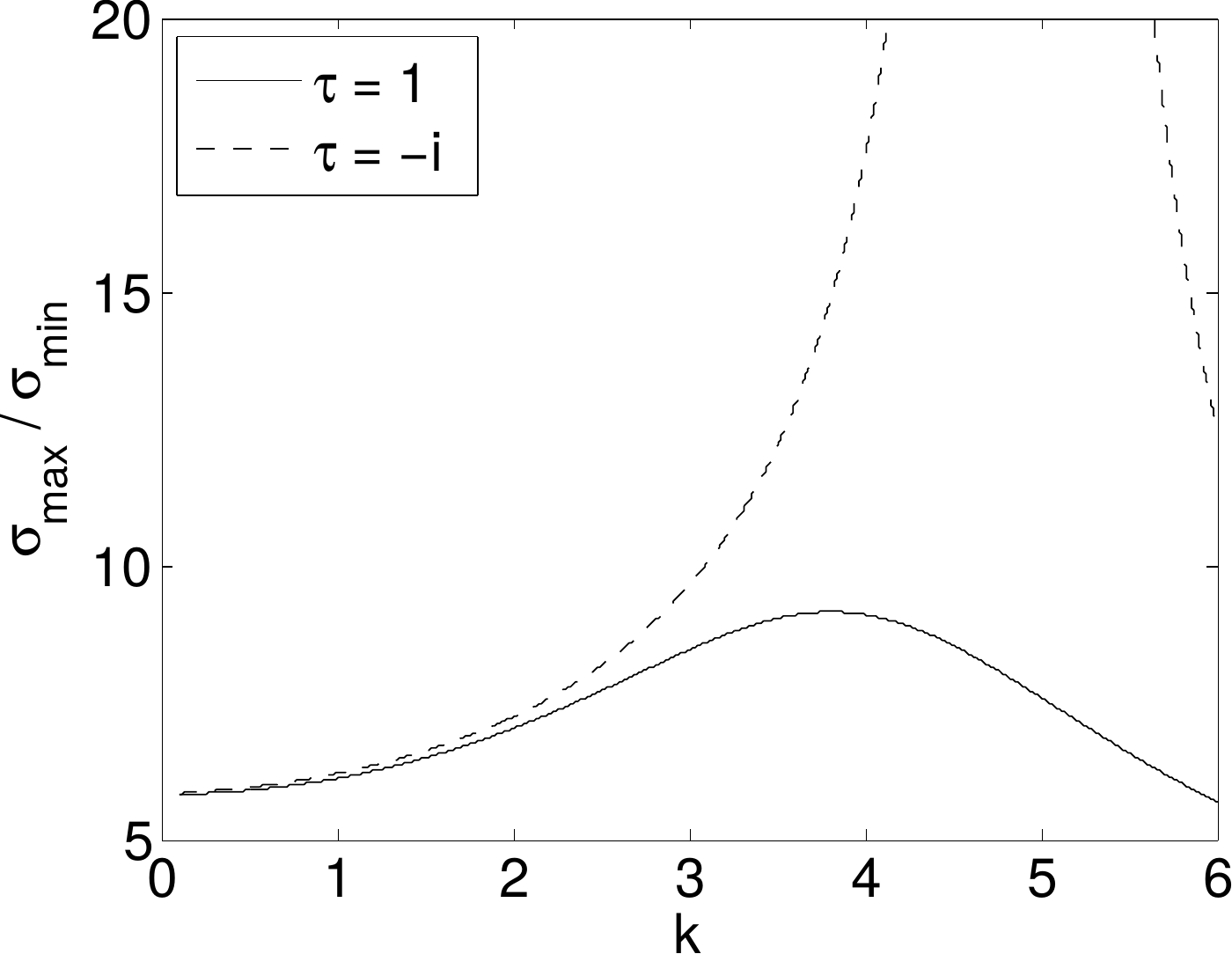}
    \caption{Degree $p=0$}
  \end{subfigure}%
  \quad
  \begin{subfigure}[b]{0.45\textwidth}
    \includegraphics[width=\textwidth]{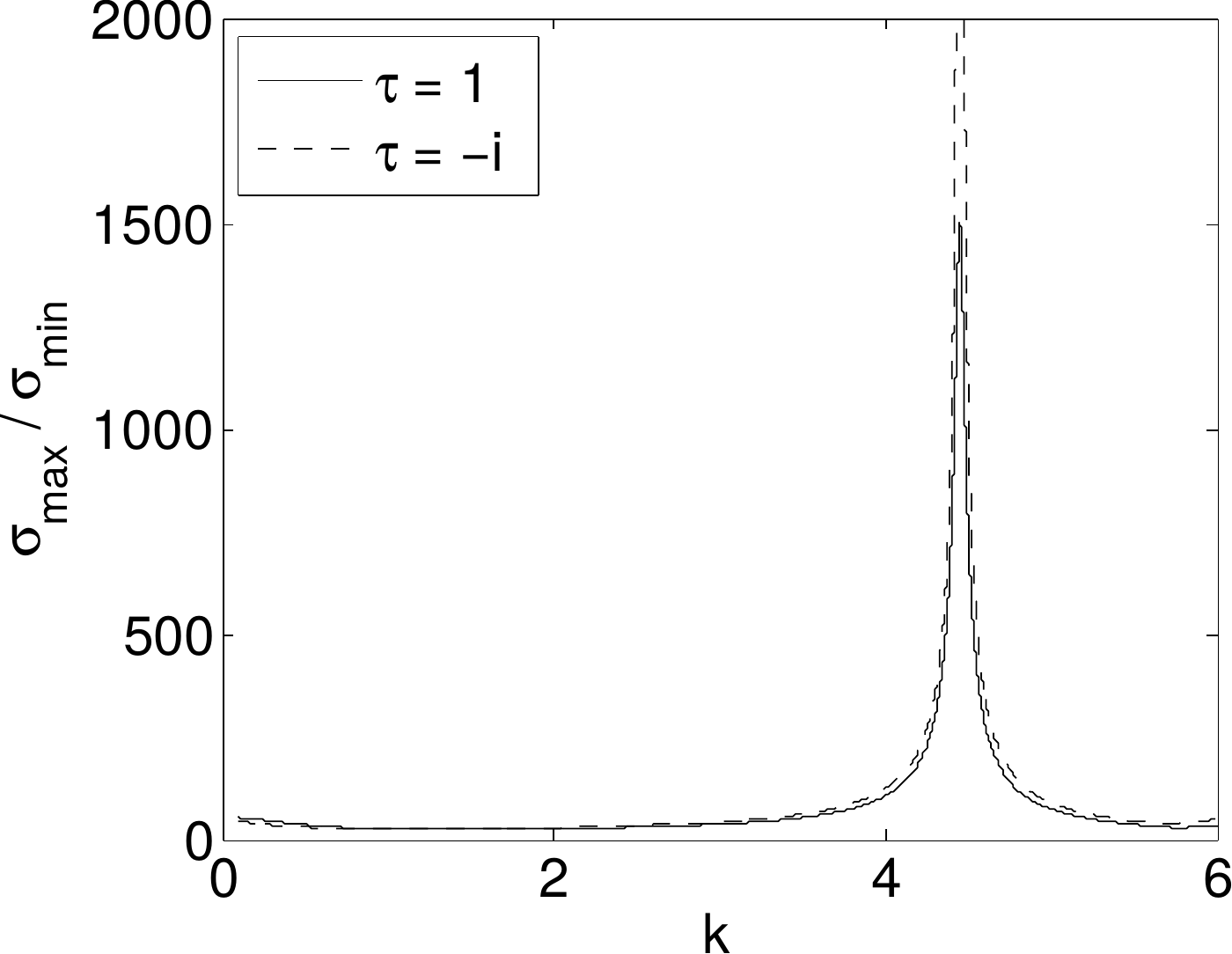}
    \caption{Degree $p=1$}
    \label{fig:conditioningp1}
  \end{subfigure}
  \caption{Conditioning of the HDG matrix for the Helmholtz
    equation near the first resonance $k=\pi\sqrt{2}\approx
    4.44$.
  }
  \label{fig:conditioning}
\end{figure}

\section*{Results of dispersion analysis for real wavenumbers}

When the wavenumber $k$ is {\em complex}, we have seen that it is
important to choose the stabilization parameter~$\tau$ such
that~\eqref{eq:conditionb} holds. We have also seen that when $k$ is
real, the stability obtained by~\eqref{eq:conditiona} is so tenuous
that it is of negligible practical value. For real wavenumbers, it is
safer to rely on stability of the (un-discretized) boundary value
problem, rather than the stability obtained by a choice of $\tau$.

The focus of this section is on {\em real} $k$ and the Helmholtz
equation~\eqref{eq:Helmholtz}.  In this case, having already separated
the issue of stability from the choice of $\tau$, we are now free to
optimize the choice of $\tau$ for other goals. By means of a
dispersion analysis, we now proceed to show that some values of $\tau$
are better than others for minimizing discrepancies in wavespeed.
Since dispersion analyses are limited to the study of propagation 
of plane waves (that solve the Helmholtz equation), we 
will not explicitly consider the Maxwell HDG system in
this section.  However, since we have written the Helmholtz and
Maxwell system consistently with respect to the stabilization
parameter (see the transition from~\eqref{eq:numflux} 
to~\eqref{eq:4} via~\eqref{eq:5}), we anticipate our results for 
the 2D Helmholtz case to
be useful for the Maxwell case also.

\subsection*{The dispersion relation in the one-dimensional case} 
\label{sec:disp1D}

Consider the HDG method~\eqref{eq:globalHelmholtz} in the lowest order
($p=0$) case in one dimension (1D) -- after appropriately interpreting
the boundary terms in~\eqref{eq:globalHelmholtz}.  We follow the
techniques of~\cite{Ainsw04} for performing a dispersion analysis.
Using a basis on a segment of size $h$ in this order
$ u_1=1, \quad \phi_1=1, \quad \hat \phi_1 =1, \quad \hat \phi_2 = 1,
$ the HDG element matrix takes the form $M = \left[\begin{smallmatrix}
    M_{11} & M_{12} \\
    M_{21} & M_{22}
\end{smallmatrix}\right]$ 
where 
\begin{align*}
M_{11}
& = 
\begin{bmatrix}
  \ii kh&0\\
  0& -\ii k h - 2\tau
\end{bmatrix}
&
M_{12} 
& = 
\begin{bmatrix}
  -1 & +1 \\
  \tau & \tau 
\end{bmatrix}
\\
M_{21} 
& = M_{12}^t
&
M_{22} 
& =  
\begin{bmatrix}
  -\tau & 0 \\
  0 & -\tau 
\end{bmatrix}.
\end{align*}
The Schur complement for the two endpoint basis functions
$\{ \hat \phi_1, \hat \phi_2\}$ is then
\[
S = -
\begin{bmatrix}
  \displaystyle \frac{1}{\ii kh} - \frac{\tau^2}{\ii kh + 2 \tau} + \tau 
  & \displaystyle -\frac{1}{\ii kh} - \frac{\tau^2}{\ii kh + 2 \tau}
  \\[2ex]
  \displaystyle -\frac{1}{\ii kh} - \frac{\tau^2}{\ii kh + 2 \tau}
   & \displaystyle \frac{1}{\ii kh} - \frac{\tau^2}{\ii kh + 2 \tau} + \tau 
\end{bmatrix}.
\]
Applying this matrix on an infinite uniform grid (of elements of size
$h$), we obtain the stencil at an arbitrary point. If $\hat\psi_j$
denotes the solution (trace) value at the $j$th point ($j \in \ZZZ$),
then the $j$th equation reads
\[
2\hat \psi_{j} \left( 
  \frac{1}{\ii kh} - \frac{\tau^2}{\ii kh + 2 \tau} + \tau 
\right)
+ (\hat \psi_{j-1} + \hat \psi_{j+1})
\left( -\frac{1}{\ii kh} - \frac{\tau^2}{\ii kh + 2 \tau}
\right) = 0.
\]

In a dispersion analysis, we are interested in how this equation
propagates plane waves on the infinite uniform grid.  Hence,
substituting $\hat \psi_j = \exp( \ii k^h jh)$, we get the following
dispersion relation for the unknown discrete wavenumber $k^h$:
\[
\cos( k^h h) 
\left( \frac{1}{\ii kh} + \frac{\tau^2}{\ii kh + 2 \tau}
\right) 
= \left( 
  \frac{1}{\ii kh} - \frac{\tau^2}{\ii kh + 2 \tau} + \tau 
\right)
\]
Simplifying, 
\begin{equation}
  \label{eq:dispersion1D}
k^h h = 
\cos^{-1} \left( 
1 - 
\frac{ (kh)^{2} }
     {2 + \ii kh ( \tau+\tau^{-1})}
\right).
\end{equation}
This is the {\em dispersion relation} for the HDG method in the lowest
order case in one dimension. Even when $\tau$ and $k$ are real, the
argument of the arccosine is not. Hence
\begin{equation}
  \label{eq:imkh}
  \im(k^h) \ne 0,
\end{equation}
in general, indicating the {\em presence of artificial dissipation in
  HDG methods}. Note however that if $\tau$ is purely imaginary and
$kh$ is sufficiently small,~\eqref{eq:dispersion1D} implies that
$\im(k^h)=0$.

Let us now study the case of small $kh$ (i.e., large number of elements per
wavelength).  As $ kh \to 0$, using the approximation
$\cos^{-1} ( 1 - x^2/2) \approx x + x^3/24 + \cdots$ valid for small $x$,
and simplifying~\eqref{eq:dispersion1D}, we obtain 
\begin{equation}
  \label{eq:kdiff}
  k^h h - kh =
   - \frac{ ( \tau^2 +1 ) \ii } {4\tau} (kh)^2
  + O( (kh)^3).
\end{equation}
Comparing this with the discrete dispersion relation of the standard
finite element method in one space dimension (see~\cite{Ainsw04}),
namely $k^hh - kh \approx O((kh)^3)$, we find that wavespeed
discrepancies from the HDG method can be larger depending on the value
of $\tau$. In particular, we conclude that {\em if we choose
  $\tau = \pm \ii$, then the error $k^hh - kh$ in both methods are of the
  same order $O((kh)^3)$.}

Before concluding this discussion of the one-dimensional case, we note
an alternate form of the dispersion relation suitable for comparison
with later formulas. Using the half-angle formula,
equation~\eqref{eq:dispersion1D} can be rewritten as 
\begin{equation}
  \label{eq:dispersion1Dhalf}
c^2  = 1- \left(\frac{ (kh)^2}{2}\right)
\left(\frac{ \tau }{\ii kh( \tau^2+1) + 2\tau }\right),
\end{equation}
where $ c = \cos(k^h h/2).$

\subsection*{Lowest order two-dimensional case}

In the 2D case, we use an infinite grid of square elements of side
length~$h$. The HDG element matrix associated to the lowest order
($p=0$) case of~\eqref{eq:globalHelmholtz} is now larger, but the
Schur complement obtained after condensing out all interior degrees of
freedom is only $4 \times 4$ because there is one degree of freedom per
edge. Note that horizontal and vertical edges represent two distinct
types of degrees of freedom, as shown in Figures~\ref{fig:stencilA}
and~\ref{fig:stencilB}. Hence there are two types of stencils.

For conducting dispersion analysis with multiple stencils, we follow
the approach in~\cite{DeraeBabusBouil99} (described more generally in
the next subsection). Accordingly, let $C_1$ and $C_2$ denote the
infinite set of stencil centers for the two types of stencils
present in our case.  Then, we get an infinite system of equations for
the unknown solution (numerical trace) values $\hat\psi_{1,\vp_1}$ and
$\hat \psi_{2,\vp_2}$ at all $\vp_1 \in C_1$ and $\vp_2 \in C_2$,
respectively. We are interested in how this infinite system propagates
plane wave solutions in every angle $\theta$. Therefore, with the
ansatz $\hat\psi_{j,\vp_j}=a_j \exp( \ii \vec\kappa_h \cdot \vp_j)$
for constants $a_j$ ($j=1$ and $2$), where the discrete wave vector is
given by
\[
\vec \kappa_h = k^h
\begin{bmatrix}
  \cos\theta \\ \sin \theta
\end{bmatrix}
\]
we proceed to find the relation between the discrete wavenumber $k^h$
and the exact wavenumber~$k$.

Substituting the ansatz into the infinite system of equations and
simplifying, we obtain a $2 \times 2$ system $F [
\begin{smallmatrix}
  a_1 \\ a_2
\end{smallmatrix}
] = 0 $ where
\[
F = \begin{bmatrix} 2\, kh \tau^2 c_1\,c_2 &d_1\, \left( 4\,\tau+\ii
    kh \right) +2\, kh \tau^2 {c_1}^{2}
  \\
  d_2\, \left( 4\,\tau+\ii kh\right) +2\, kh \tau^2 {c_2}^{2} & 2\, kh
  \tau^2 c_1 c_2
\end{bmatrix}
\]
and, for $j=1,2$,
\begin{align}
  \label{eq:8}
  c_j & = \cos\left( \frac 1 2 h \kh_j \right), & d_j & = 2\ii
  (1-c_j^2) - \tau kh, & k_1^h & = k^h \cos \theta, & k_2^h & = k^h
  \sin \theta.
\end{align}
Hence the {\em 2D dispersion relation} relating $k^h$ to $k$ in the
HDG method is
\begin{equation}
  \label{eq:2DdispDetF}
  \det(F)=0.  
\end{equation}

To formally compare this to the 1D dispersion relation, consider these
two sufficient conditions for $\det (F)=0$ to hold:
\begin{equation}
  \label{eq:suff}
  2(kh)^2\tau^2  c_j^2
  + d_j\left(
    2\tau kh +\ii (k_jh)^2 
  \right)= 0,
  \qquad \text{ for } j =1,2,
\end{equation}
where $k_1 = k \cos\theta$ and $k_2 = k \sin\theta$. (Indeed,
multiplying~\eqref{eq:suff}$_j$ by $d_{j+1}$ ($j=1$) or $d_{j-1}$ ($j=2$) and summing over $j=1,2$, one
obtains a multiple of $\det(F)$.) The equations in~\eqref{eq:suff} can be
simplified to
\begin{align}
  \label{eq:2Ddispersion}
  c_j^2 = 1 - \frac{ (k_jh)^2 }{2\ii} \left( \frac{\,kh\,\tau\,} {(k_j
      h)^2 +(kh)^{2}\,{\tau}^{2} - 2\,\ii\,kh\,\tau } \right), \qquad
  j=1,2,
\end{align}
which are relations that have a form similar to the 1D
relation~\eqref{eq:dispersion1Dhalf}. Hence we use asymptotic
expansions of arccosine for small $kh$, similar to the ones used in
the 1D case, to obtain an expansion for $k^h_j$, for $j=1,2$.

The final step in the calculation is the use of the simple identity
\begin{equation}
  \label{eq:7}
k^h = \bigg( (k^h_1)^2 + (k^h_2)^2 \bigg)^{1/2}.
\end{equation}
Simplifying the above-mentioned expansions for each term on the right
hand side above, we obtain
\begin{equation}
  \label{eq:9}
k^hh - kh = \frac{\ii ( \cos( 4\,\theta) +3 + 4\,{\tau}^{2} )
}{16\,\tau } \, (k h)^{2} + O( (kh)^3 )
\end{equation}
as $kh \to 0$.  Thus, we conclude that if we want dispersion errors to
be $O( (kh)^3)$, then we must choose
\begin{equation}
  \label{eq:tautheta}
\tau = \pm \frac 1 2 \ii
\sqrt{\cos( 4\theta) + 3},  
\end{equation}
a prescription that is not very useful in
practice because it depends on the propagation angle $\theta$.
However, we can obtain a more practically useful condition by setting
$\tau$ to be the constant value that best approximates $\pm \frac 1 2
\ii \sqrt{\cos( 4\theta) + 3}$ for all $0 \le \theta \le \pi/2$,
namely
\begin{equation}
  \label{eq:besttau}
  \tau = \pm \ii \frac{\sqrt 3}{2}.
\end{equation}
These values of $\tau$ {\em asymptotically minimize errors in discrete
  wavenumber over all angles for the lowest order 2D HDG method.}
Note that for any purely imaginary $\tau$, \eqref{eq:2Ddispersion}
implies that $k^h_j$ is real if $kh$ is sufficiently small, so 
\begin{equation}
  \label{eq:11}
  \im(k^h) =0,
\end{equation}
thus eliminating artificial dissipation.

\begin{figure}
  \centering
  \begin{subfigure}[b]{0.45\textwidth}
    \includegraphics[width=\textwidth]{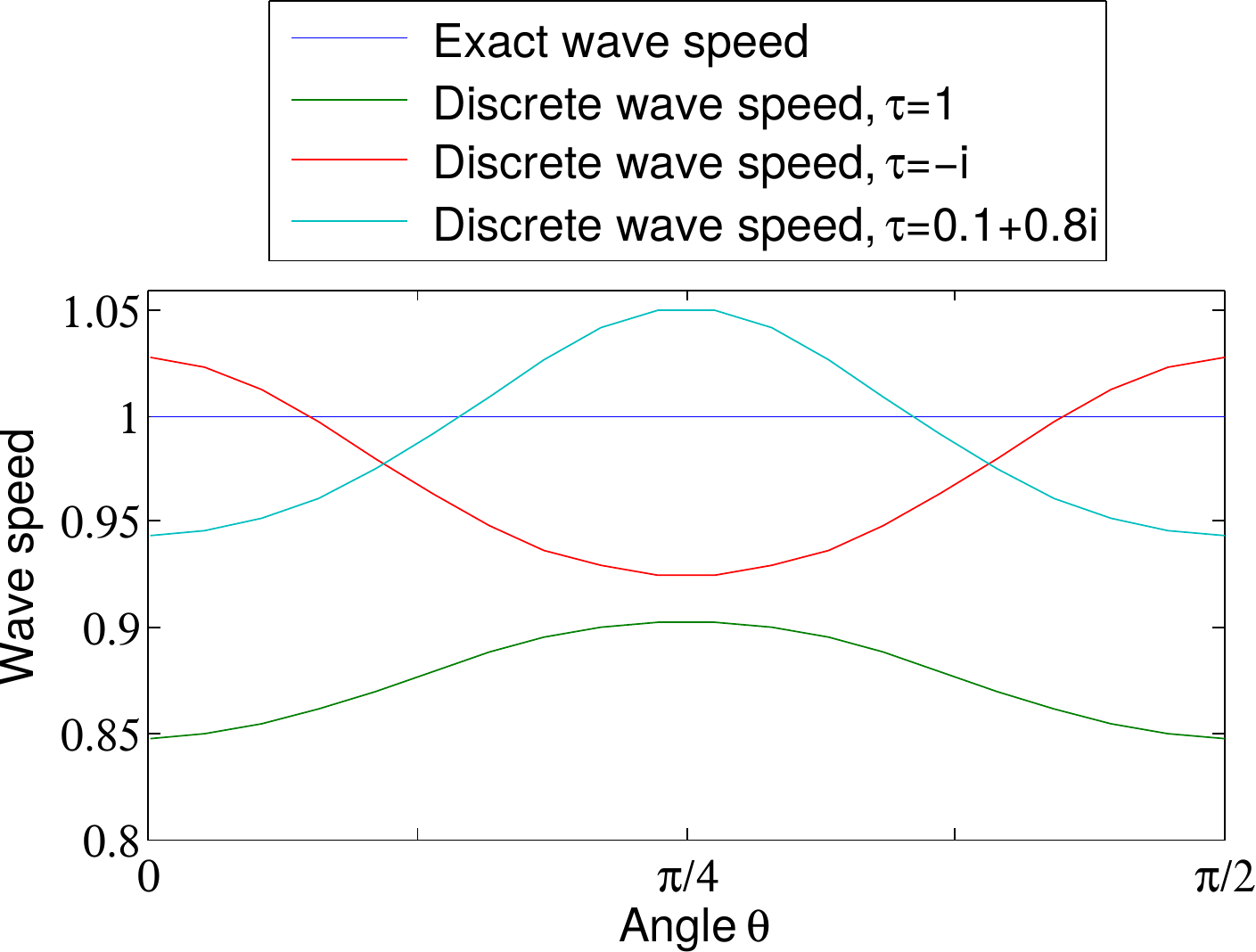}
    \caption{$p=0$}
    \label{fig:circle0}
  \end{subfigure}%
  \quad\quad
  \begin{subfigure}[b]{0.45\textwidth}
    \includegraphics[width=\textwidth]{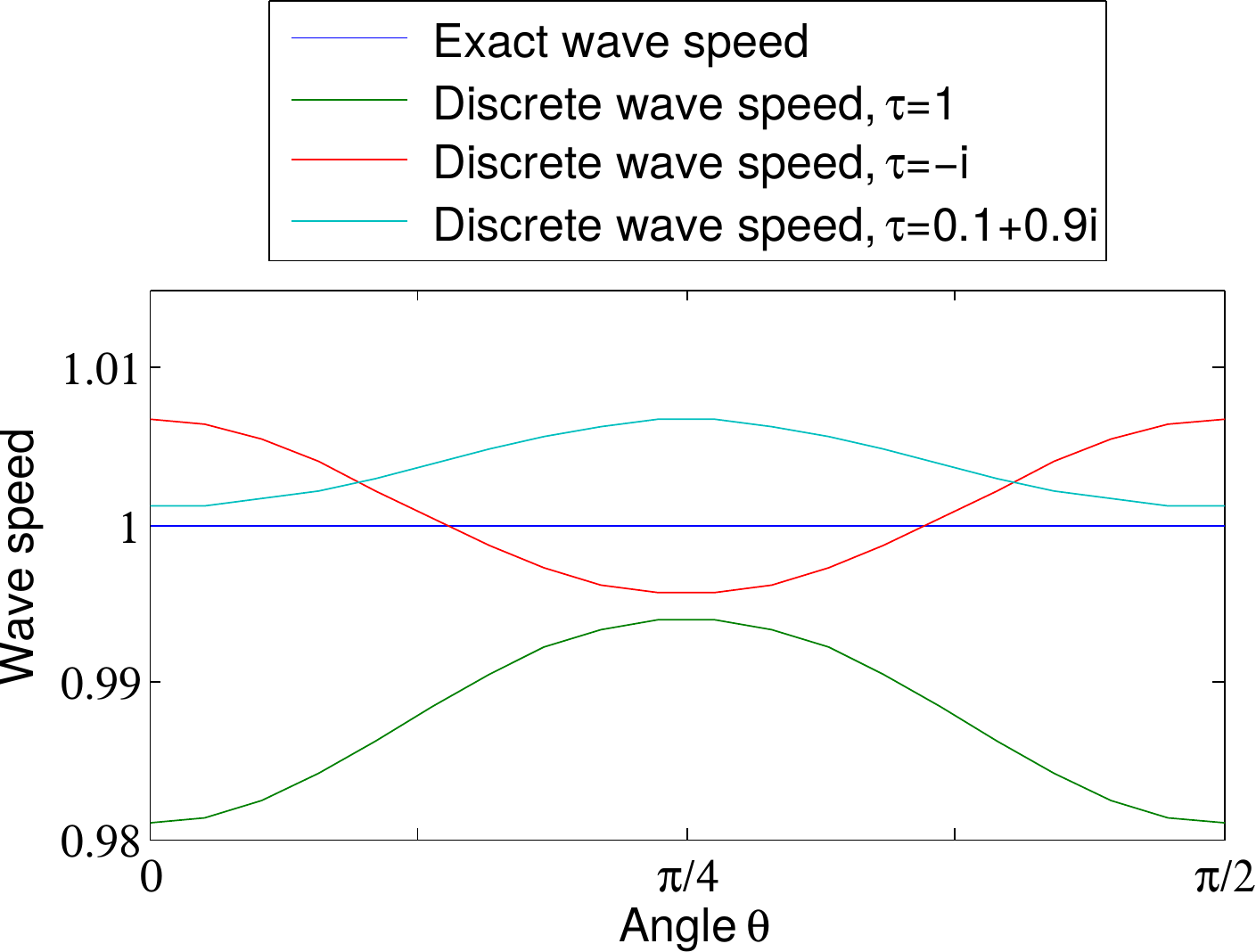}
    \caption{$p=1$}
    \label{fig:circle1}
  \end{subfigure}%
  \caption{Numerical wave speed $\re(\vec k^h(\theta))$ as a function
    of $\theta$ for various choices of $\tau$. Here, $k=1$ and
    $h=\pi/4$.}
  \label{fig:circle}
\end{figure}

We now report results of numerical computation of $k^h = k^h(\theta)$
by directly applying a nonlinear solver to the 2D dispersion
relation~\eqref{eq:2DdispDetF} (for a set of propagation angles
$\theta$). The obtained values of the real part $\re k^h(\theta)$ are
plotted in Figure~\ref{fig:circle0}, for a few fixed values of
$\tau$. The discrepancy between the exact and discrete curves
quantifies the difference in the wave speeds for the computed and the
exact wave.  Next, analyzing the computed $k^h(\theta)$ for values of
$\tau$ on a uniform grid in the complex plane, we found that the
values of $\tau$ that minimize $|kh-k^h(\theta) h|$ are purely
imaginary.  As shown in Figure~\ref{fig:tau_vs_theta}, these
$\tau$-values approach the asymptotic values determined analytically
in~\eqref{eq:tautheta}.  A second validation of our analysis is
performed by considering the maximum error over all $\theta$ for each
value of $\tau$ and then determining the practically optimal value of
$\tau$.  The results, given in Table~\ref{tab:optimtau}, show that the
optimal $\tau$ values do approach the analytically determined value
(see~\eqref{eq:besttau}) of $\pm \ii \frac{\sqrt 3}{2}\approx
\pm 0.866 \ii$.
Further numerical results for the $p=0$ case are presented together with a 
higher order case in the next subsection.

\begin{table}
  \centering
\begin{center}
\begin{tabular}{|c|c|c|}
\hline
$kh$ & Optimal $\tau$,       & Optimal $\tau$,       \\
     & $\mathrm{Im}(\tau)>0$ & $\mathrm{Im}(\tau)<0$ \\
\hline
$\pi/4$ & $0.807\ii$ & $-0.931\ii$\\
$\pi/8$ & $0.837\ii$ & $-0.898\ii$ \\
$\pi/16$ & $0.851\ii$ &$-0.882\ii$ \\
$\pi/32$ & $0.859\ii$ &$-0.874\ii$ \\
$\pi/64$ & $0.863\ii$ & $-0.871\ii$\\
$\pi/128$ & $0.865\ii$ &$-0.868\ii$ \\
$\pi/256$ & $0.866\ii$ & $-0.867\ii$\\
\hline
\end{tabular}
\end{center}

\caption{Numerically found values of $\tau$ that minimize $| kh - k^h(\theta) h|$ for all $\theta$ in the $p=0$ case.}
  \label{tab:optimtau}
\end{table}

\begin{figure}
  \centering
  \begin{subfigure}[b]{0.41\textwidth}
    \includegraphics[width=\textwidth]{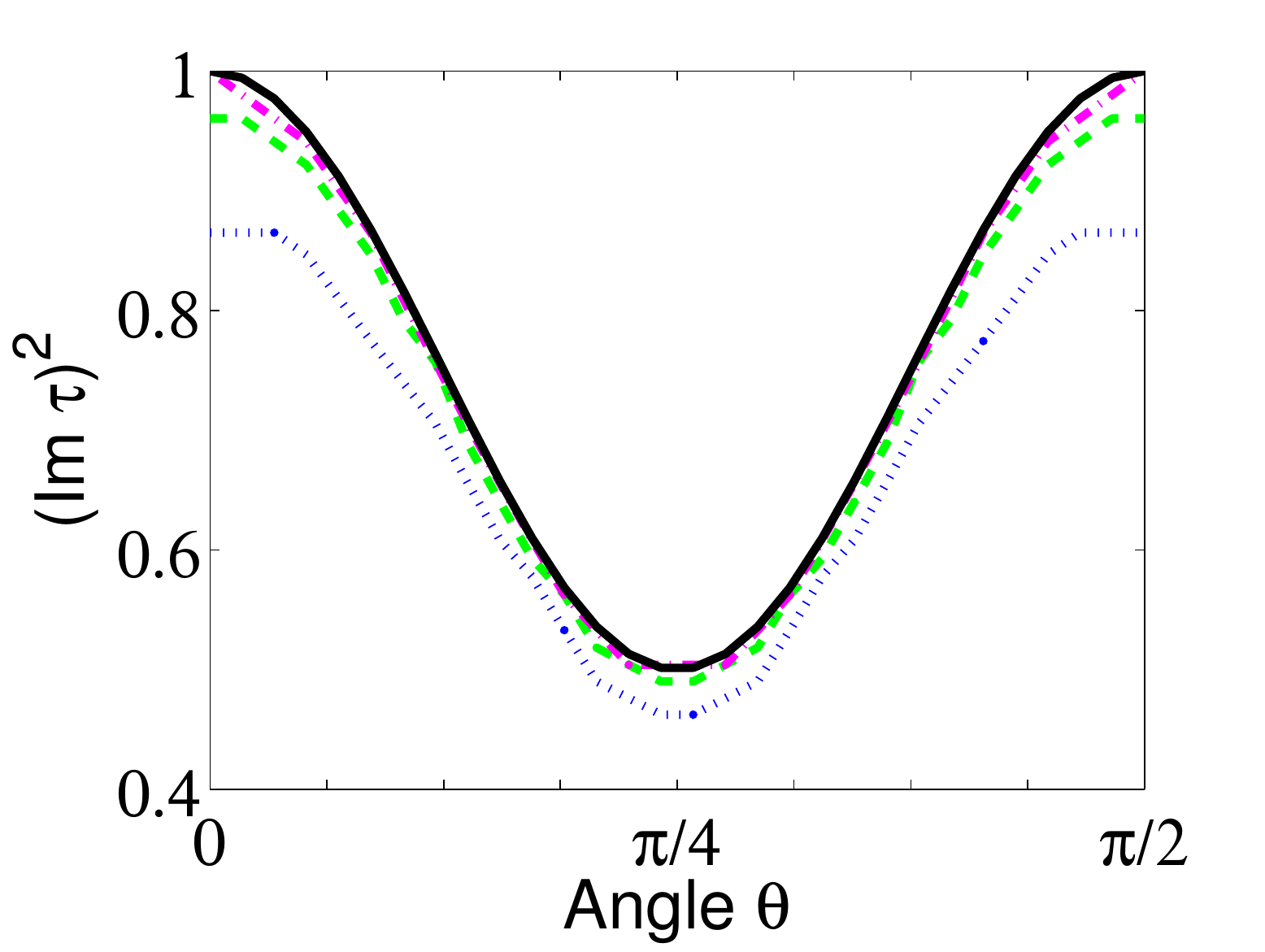}
    \caption{$\mathrm{Im}(\tau) > 0$}
    \label{fig:pos_im}
  \end{subfigure}%
  \begin{subfigure}[b]{0.18\textwidth}
    \includegraphics[width=\textwidth]{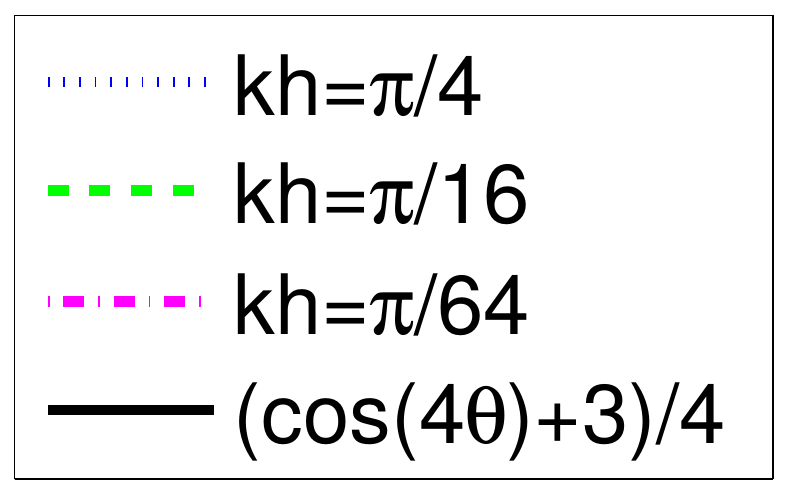}
    \vspace{1.7cm}
  \end{subfigure}%
  \begin{subfigure}[b]{0.41\textwidth}
     \includegraphics[width=\textwidth]{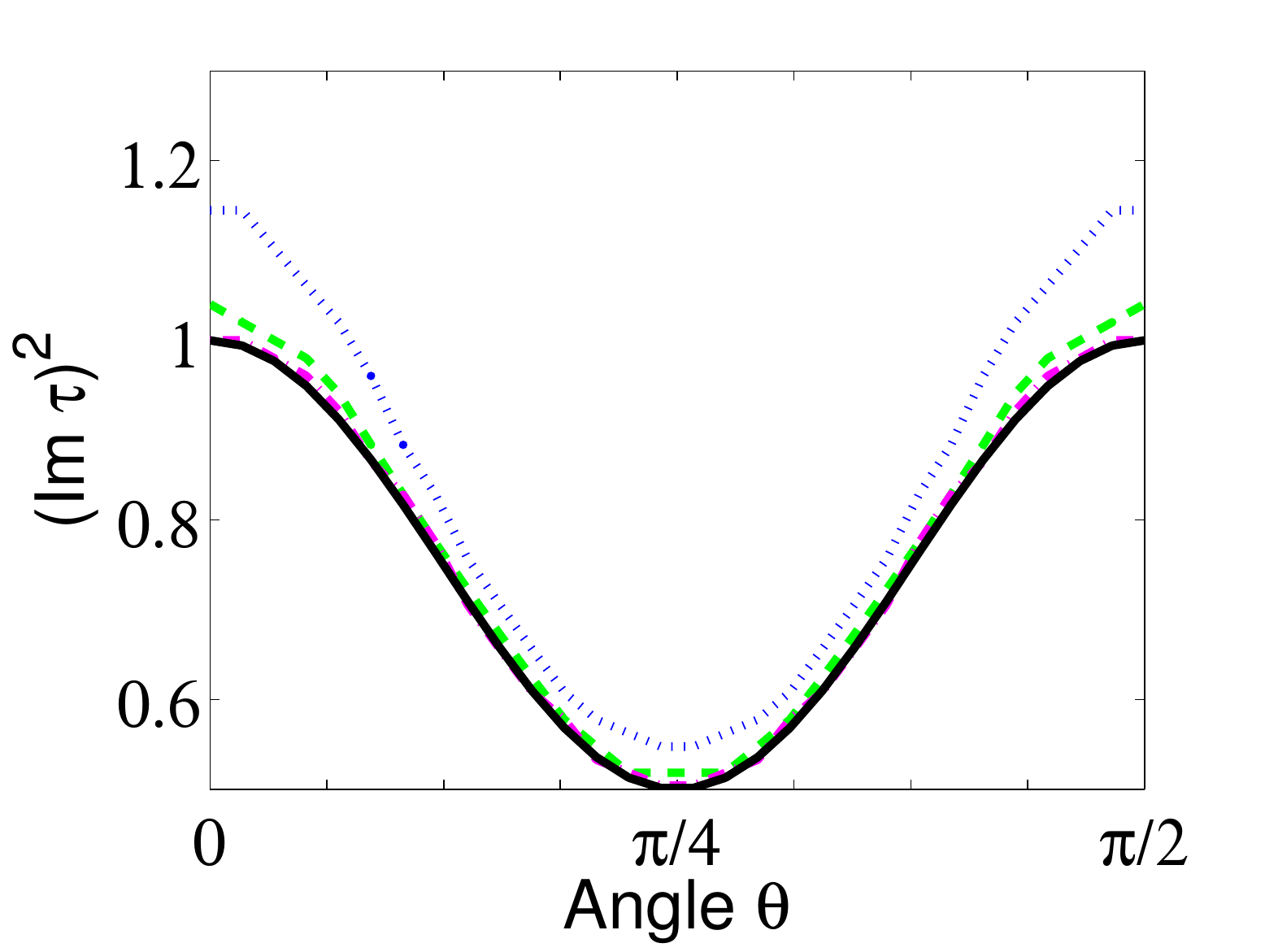}
    \caption{$\mathrm{Im}(\tau) < 0$}
    \label{fig:neg_im}
  \end{subfigure}%
  \caption{The values of $\tau$ that locally minimize $|kh-k^hh|$ are
    purely imaginary. Here, $(\mathrm{Im}(\tau))^2$ is compared with
    asymptotic values (dashed lines).}
  \label{fig:tau_vs_theta}
\end{figure}

\subsection*{Higher order case} \label{ssec:higher}

To go beyond the $p=0$ case, we extend a technique
of~\cite{DeraeBabusBouil99} (as in~\cite{GopalMugaOliva14}).
Using a higher order HDG stencil, we want to 
obtain an analogue of~\eqref{eq:2DdispDetF}, which can be 
numerically solved for 
the discrete wavenumber
$k^h=k^h(\theta)$. The accompanying dispersive, dissipative, and total errors
are defined respectively by
\begin{equation}
  \label{eq:10}
\begin{gathered}  
  \epsilon_{\text{disp}} = \max_{\theta}|\re(k^h(\theta))-k|,\qquad
  \epsilon_{\text{dissip}} = \max_{\theta}|\im(k^h(\theta))|,\;
  \\
  \epsilon_{\text{total}} = \max_{\theta}|k^h(\theta)-k|.
\end{gathered}  
\end{equation}

\begin{figure}
  \begin{subfigure}[b]{0.3\textwidth}
    \centering
    \begin{tikzpicture}
      \draw [step = 1cm,dotted] (0,0) grid (3,4); \draw [step = 1cm]
      (1,1) grid (2,3); \foreach \x in {1,2} { \fill[white] (\x,1.5)
        circle (2.5pt); \fill[white] (\x,2.5) circle (2.5pt); }
      \foreach \y in {1,2,3} { \fill[white] (1.5,\y) circle (2.5pt); }
      \foreach \x in {1,2} { \draw (\x,1.5) circle (2.5pt); \draw
        (\x,2.5) circle (2.5pt); } \foreach \y in {1,2,3} { \draw
        (1.5,\y) circle (2.5pt); } \fill (1.5,2) circle (2.5pt);
    \end{tikzpicture}
    \caption{}
    \label{fig:stencilA}
  \end{subfigure}
  \begin{subfigure}[b]{0.3\textwidth}
    \centering
    \begin{tikzpicture}
      \draw [step = 1cm,dotted] (0,0) grid (4,3); \draw [step = 1cm]
      (1,1) grid (3,2); \foreach \x in {1,2,3} { \fill[white] (\x,1.5)
        circle (2.5pt); } \foreach \y in {1,2} {\fill[white] (1.5,\y)
        circle (2.5pt); \fill[white] (2.5,\y) circle (2.5pt); }
      \foreach \x in {1,2,3} { \draw (\x,1.5) circle (2.5pt); }
      \foreach \y in {1,2} {\draw (1.5,\y) circle (2.5pt); \draw
        (2.5,\y) circle (2.5pt); } \fill (2,1.5) circle (2.5pt);
    \end{tikzpicture}
    \caption{}
    \label{fig:stencilB}
  \end{subfigure}
  \begin{subfigure}[b]{0.3\textwidth}
    \centering
    \begin{tikzpicture}
      \draw [step = 1cm,dotted] (0,0) grid (3,4); \draw [step = 1cm]
      (1,1) grid (2,3); \foreach \x in {1,2} { \fill[white] (\x,1.09)
        circle (2.5pt); \fill[white] (\x,1.91) circle (2.5pt);
        \fill[white] (\x,2.09) circle (2.5pt);\fill[white] (\x,2.91)
        circle (2.5pt);} \foreach \y in {1,2,3} { \fill[white]
        (1.09,\y) circle (2.5pt); \fill[white] (1.91,\y) circle
        (2.5pt);} \foreach \x in {1,2} { \draw (\x,1.09) circle
        (2.5pt); \draw (\x,1.91) circle (2.5pt); \draw (\x,2.09)
        circle (2.5pt);\draw (\x,2.91) circle (2.5pt);} \foreach \y in
      {1,2,3} { \draw (1.09,\y) circle (2.5pt); \draw (1.91,\y) circle
        (2.5pt);} \fill (1.09,2) circle (2.5pt);
    \end{tikzpicture}
    \caption{}
    \label{fig:stencilC}
  \end{subfigure}
  \begin{subfigure}[b]{0.3\textwidth}
    \centering
    \begin{tikzpicture}
      \draw [step = 1cm,dotted] (0,0) grid (3,4); \draw [step = 1cm]
      (1,1) grid (2,3); \foreach \x in {1,2} { \fill[white] (\x,1.09)
        circle (2.5pt); \fill[white] (\x,1.91) circle (2.5pt);
        \fill[white] (\x,2.09) circle (2.5pt);\fill[white] (\x,2.91)
        circle (2.5pt);} \foreach \y in {1,2,3} { \fill[white]
        (1.09,\y) circle (2.5pt); \fill[white] (1.91,\y) circle
        (2.5pt);} \foreach \x in {1,2} { \draw (\x,1.09) circle
        (2.5pt); \draw (\x,1.91) circle (2.5pt); \draw (\x,2.09)
        circle (2.5pt);\draw (\x,2.91) circle (2.5pt);} \foreach \y in
      {1,2,3} { \draw (1.09,\y) circle (2.5pt); \draw (1.91,\y) circle
        (2.5pt);} \fill (1.91,2) circle (2.5pt);
    \end{tikzpicture}
    \caption{}
    \label{fig:stencilD}
  \end{subfigure}
  \begin{subfigure}[b]{0.3\textwidth}
    \centering
    \begin{tikzpicture}
      \draw [step = 1cm,dotted] (0,0) grid (4,3); \draw [step = 1cm]
      (1,1) grid (3,2); \foreach \x in {1,2,3} { \fill[white]
        (\x,1.09) circle (2.5pt); \fill[white] (\x,1.91) circle
        (2.5pt);} \foreach \y in {1,2} {\fill[white] (1.09,\y) circle
        (2.5pt); \fill[white] (1.91,\y) circle (2.5pt); \fill[white]
        (2.09,\y) circle (2.5pt); \fill[white] (2.91,\y) circle
        (2.5pt);} \foreach \x in {1,2,3} { \draw (\x,1.09) circle
        (2.5pt); \draw (\x,1.91) circle (2.5pt);} \foreach \y in {1,2}
      {\draw (1.09,\y) circle (2.5pt); \draw (1.91,\y) circle (2.5pt);
        \draw (2.09,\y) circle (2.5pt); \draw (2.91,\y) circle
        (2.5pt);} \fill (2,1.09) circle (2.5pt);
    \end{tikzpicture}
    \caption{}
    \label{fig:stencilE}
  \end{subfigure}
  \begin{subfigure}[b]{0.3\textwidth}
    \centering
    \begin{tikzpicture}
      \draw [step = 1cm,dotted] (0,0) grid (4,3); \draw [step = 1cm]
      (1,1) grid (3,2); \foreach \x in {1,2,3} { \fill[white]
        (\x,1.09) circle (2.5pt); \fill[white] (\x,1.91) circle
        (2.5pt);} \foreach \y in {1,2} {\fill[white] (1.09,\y) circle
        (2.5pt); \fill[white] (1.91,\y) circle (2.5pt); \fill[white]
        (2.09,\y) circle (2.5pt); \fill[white] (2.91,\y) circle
        (2.5pt);} \foreach \x in {1,2,3} { \draw (\x,1.09) circle
        (2.5pt); \draw (\x,1.91) circle (2.5pt);} \foreach \y in {1,2}
      {\draw (1.09,\y) circle (2.5pt); \draw (1.91,\y) circle (2.5pt);
        \draw (2.09,\y) circle (2.5pt); \draw (2.91,\y) circle
        (2.5pt);} \fill (2,1.91) circle (2.5pt);
    \end{tikzpicture}
    \caption{}
    \label{fig:stencilF}
  \end{subfigure}
  \caption{Stencils corresponding to the shaded node
    types. (\textsc{A})--(\textsc{B}): Two node types of the lowest
    order ($p=0$) method; (\textsc{C})--(\textsc{F}): Four node types
    of the first order ($p=1$) method.}
  \label{fig:stencils}
\end{figure}
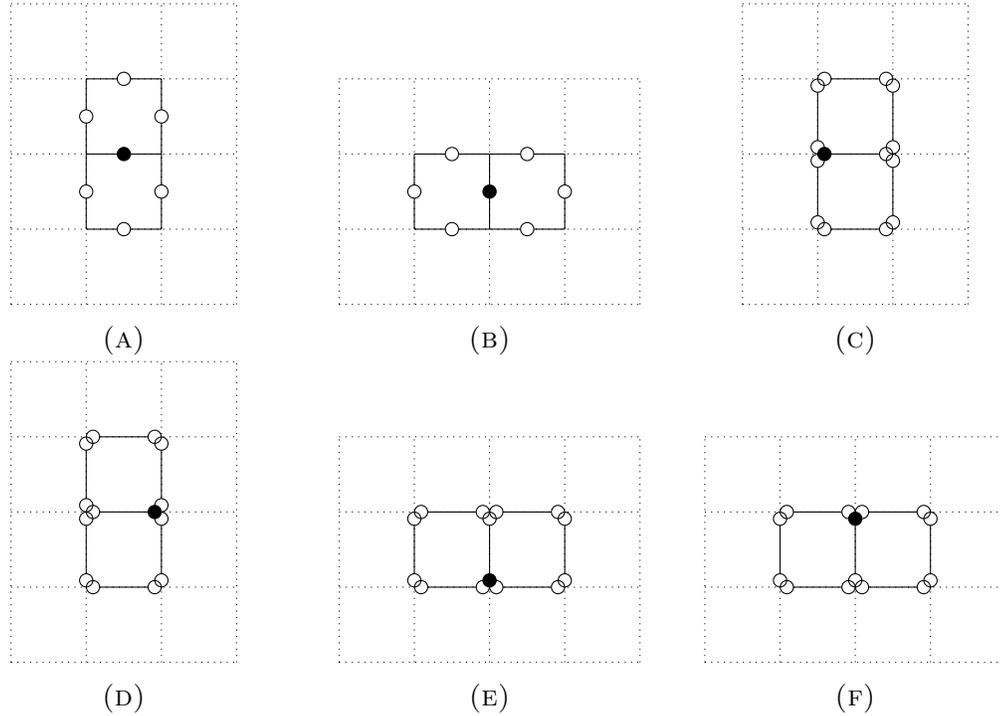

Again, we consider an infinite lattice of $h\times h$ square elements
with the ansatz that the HDG degrees of freedom interpolate a
plane wave traveling in the $\theta$ direction with wavenumber
$k^h$. 
The lowest order and next higher order HDG stencils are compared in
Figure~\ref{fig:stencils}. Note that the figure only 
shows the interactions of the degrees of
freedom corresponding to the $\hat \phi$ variable---the only degrees
of freedom involved after elimination of the $\vu$ and $\phi$ degrees
of freedom via static condensation. The lowest order method has two
node types (shown in Figures~\ref{fig:stencilA}--\ref{fig:stencilB}), 
while  the first order method has four node types (shown 
in Figures~\ref{fig:stencilC}--\ref{fig:stencilF}). For a
method with $S$ distinct node types, denote the solution value at a
node of the $s^{th}$ type, $1\leq s \leq S$, located at $\vl
h\in\mathbb{R}^2$, by $\psi_{s,\vl}$. With our ansatz that these
solution values interpolate a plane wave, we have
\[
\psi_{s,\vl}=a_se^{\ii \vec k^h \cdot\vl h},
\]
for some constants $a_s$.

Now, to develop notation to express each stencil's equation, we fix a
stencil within the lattice. Suppose that it corresponds to a node of
the $t^{th}$ type, $1\leq t\leq S$, that is located at $\vj h$. For
$1\leq s \leq S$, define $J_{t,s}=\{\vl \in \mathbb{R}^2 ~:~ \text{a
  node of type } s \text{ is located at } (\vj + \vl)h\}$ and, for
$\vl\in J_{t,s}$, denote the stencil coefficient of the node at
location $(\vj + \vl) h$ by $D_{t,s,\vl}$. The stencil coefficient is
the linear combination of the condensed local matrix entries that
would likewise appear in the global matrix of
equation~\eqref{eq:reducedHelmholtz}. Both it and the set $J_{t,s}$
are translation invariant, i.e., independent of $\vj$. Since
plane waves are exact solutions to the Helmholtz equation with zero
sources, the stencil's equation is
\[
\sum_{s=1}^S~\sum_{l\in J_{t,s}}D_{t,s,\vl}~\psi_{s,\vj+\vl}=0.
\]
Finally, we remove all dependence on $\vj$ in this equation by
dividing by $e^{\ii\vec k^h\cdot\vj h}$, so there are $S$ equations in
total, with the $t^{th}$ equation given by
\begin{equation}\label{eq:stencileq}
  \sum_{s=1}^S a_s\sum_{l\in J_{t,s}}D_{t,s,\vl}~e^{\ii \vec k^h
    \cdot\vl h}=0. 
\end{equation}
Defining the $S\times S$ matrix $F(k^h)$ by
\[ [F(k^h)]_{t,s}=\sum_{l\in J_{t,s}}D_{t,s,\vl}~e^{\ii
  k^h[\cos\theta, \sin\theta] \cdot\vl h},
\]
we observe that non-trivial coefficients $\{a_s\}$ exist if and only
if $k^h$ is such that
\begin{equation}\label{eq:det}
  \text{det}(F(k^h))=0.
\end{equation}
This is the equation that we solve to determine $k^h$ for a given
$\theta$ for any order (cf.~\eqref{eq:2DdispDetF}).

\begin{figure}
  \centering
  \begin{subfigure}[b]{0.45\textwidth}
    \includegraphics[width=\textwidth]{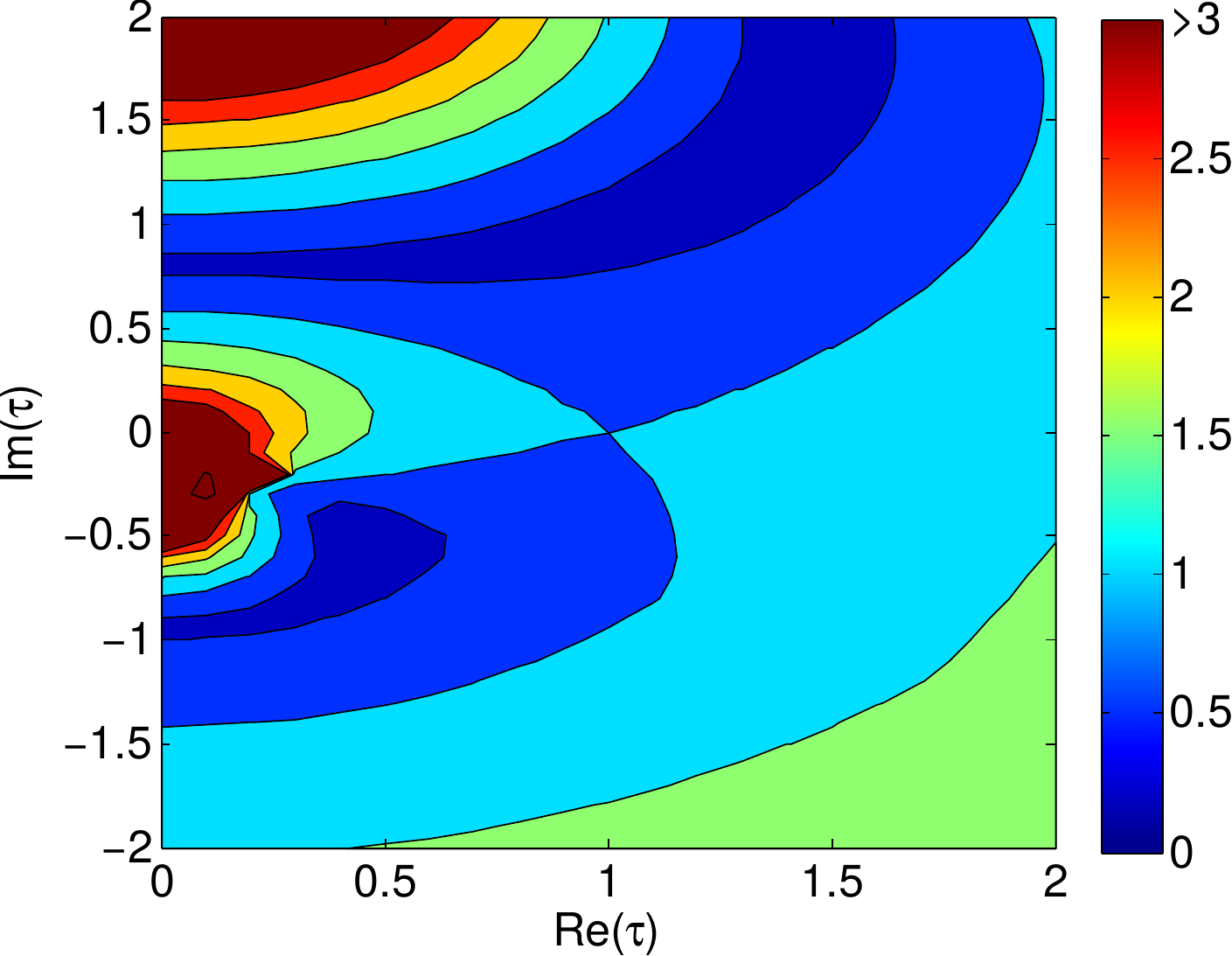}
    \caption{Dispersive error, $p=0$}
    \label{fig:disperr0}
  \end{subfigure}%
  \quad\quad
  \begin{subfigure}[b]{0.45\textwidth}
    \includegraphics[width=\textwidth]{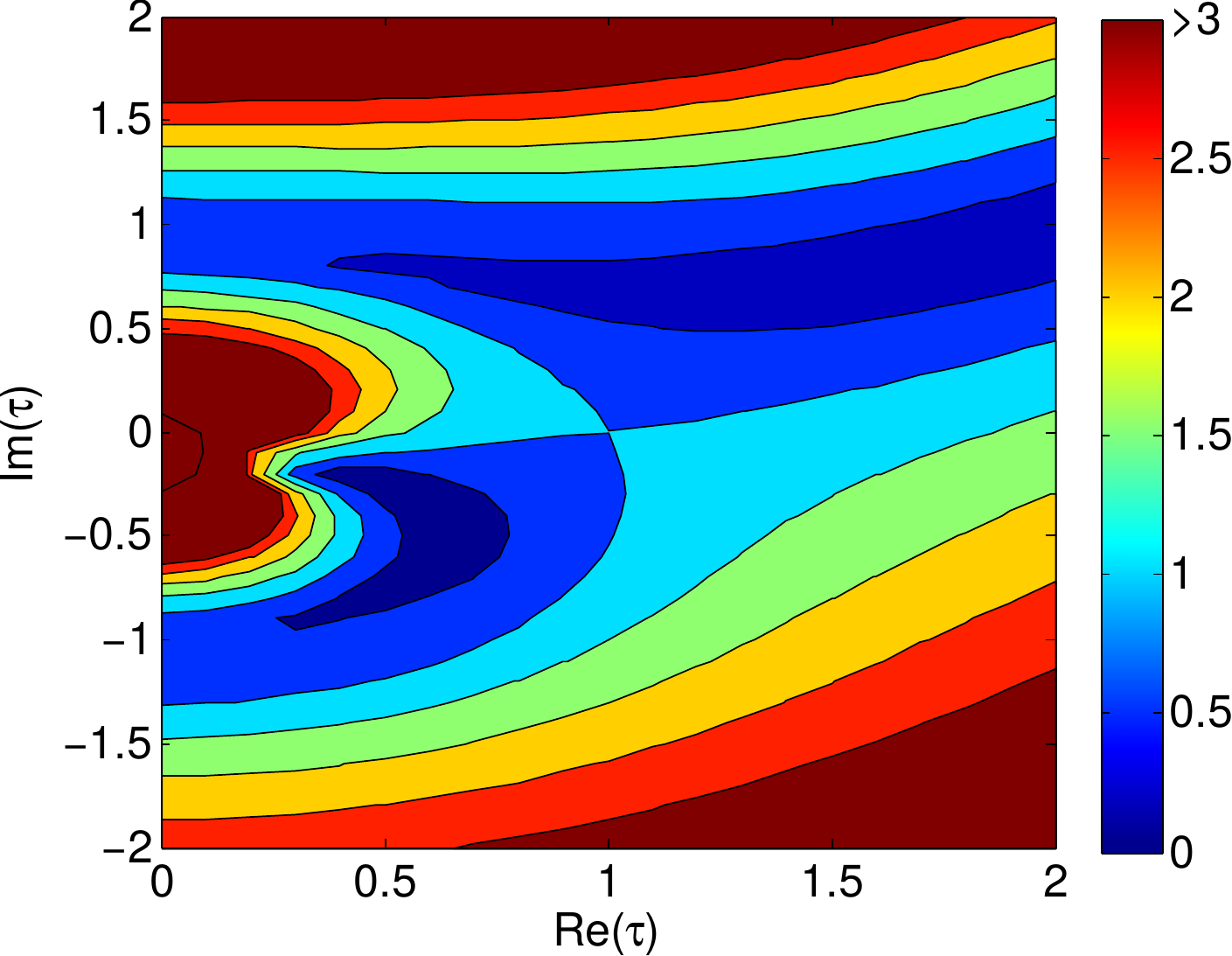}
    \caption{Dispersive error, $p=1$}
    \label{fig:disperr1}
  \end{subfigure}%

  \vspace{0.2cm}
  \begin{subfigure}[b]{0.45\textwidth}
    \includegraphics[width=\textwidth]{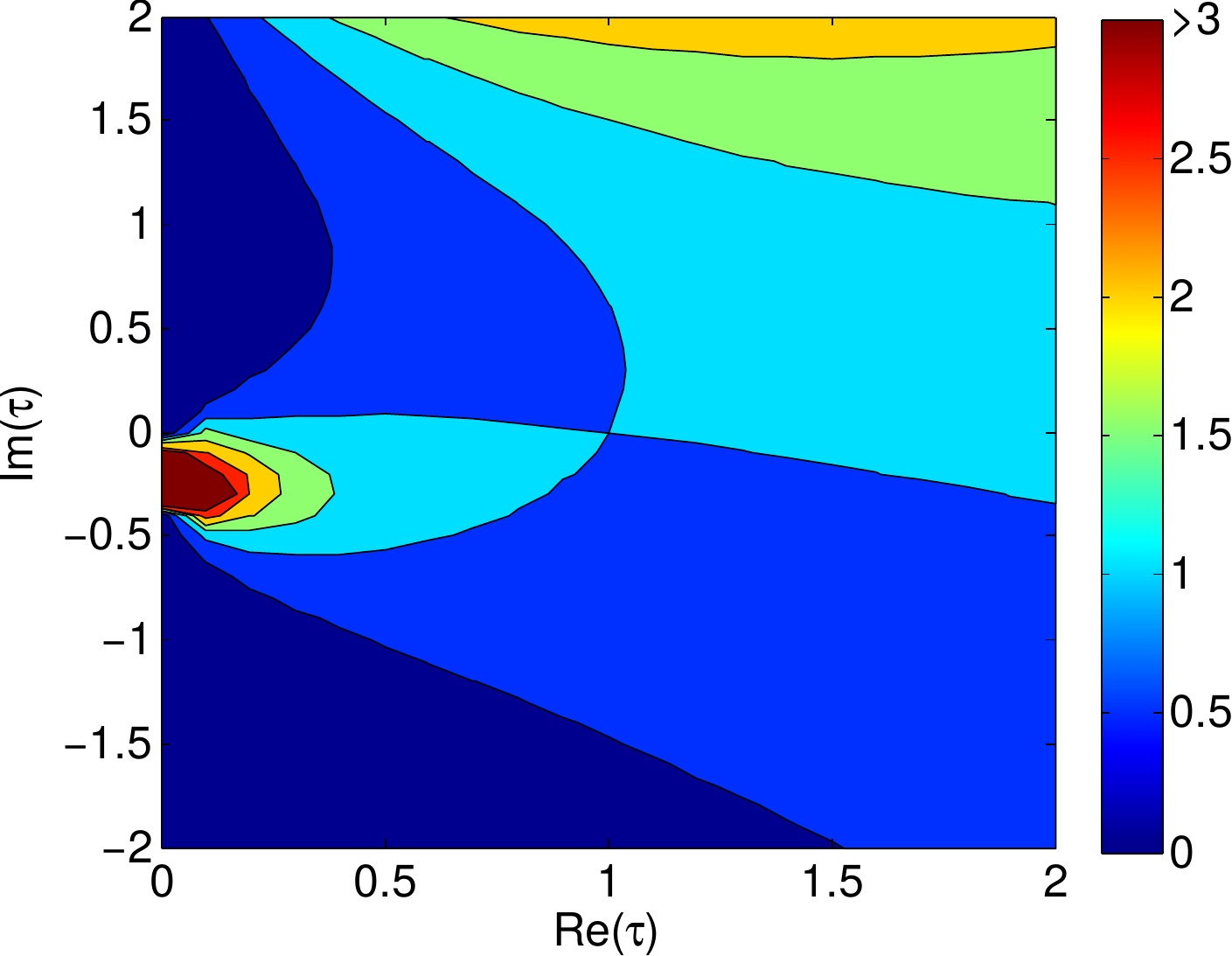}
    \caption{Dissipative error, $p=0$}
    \label{fig:dissiperr0}
  \end{subfigure}
  \quad\quad
  \begin{subfigure}[b]{0.45\textwidth}
    \includegraphics[width=\textwidth]{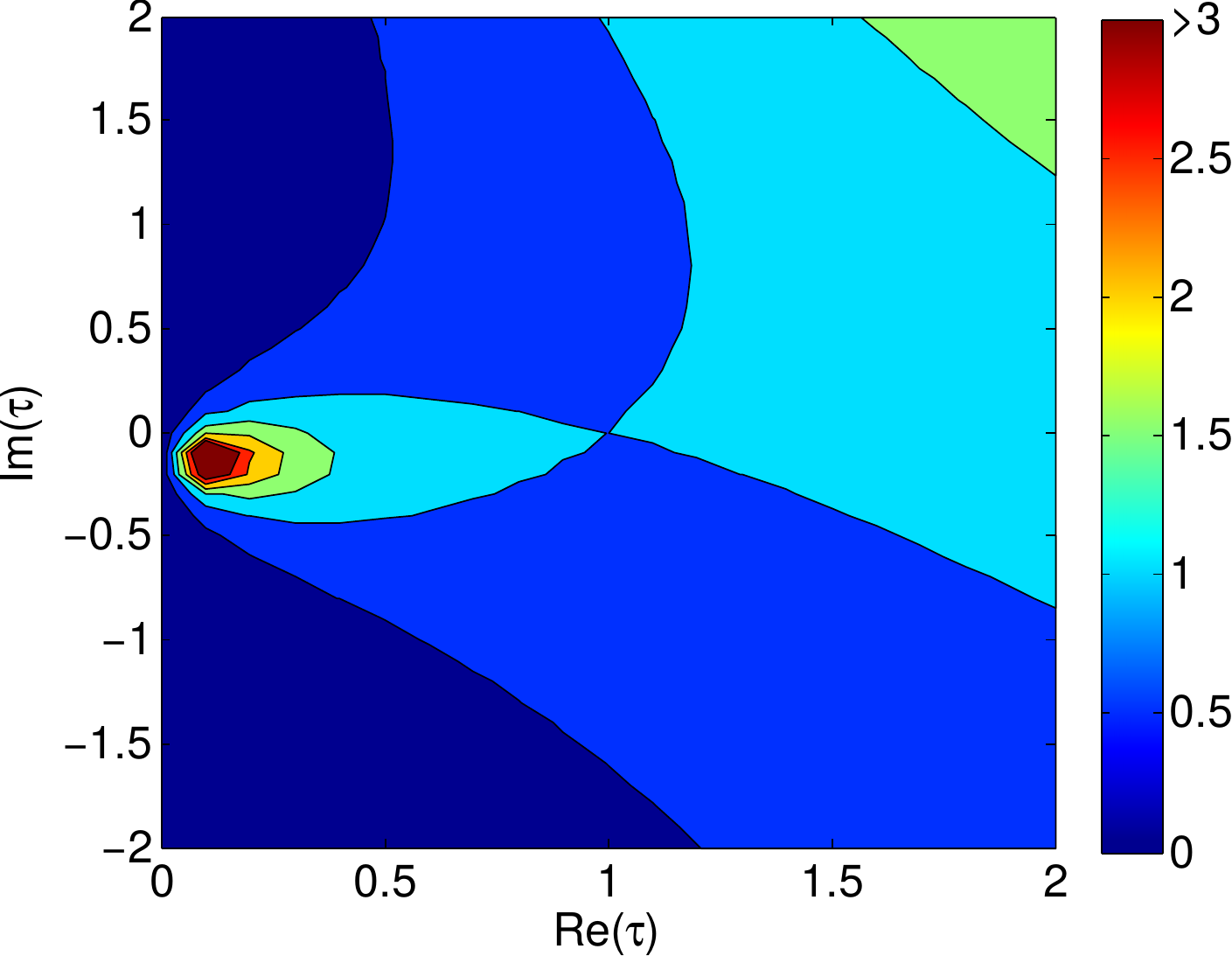}
    \caption{Dissipative error, $p=1$}
    \label{fig:dissiperr1}
  \end{subfigure}

  \vspace{0.2cm}
  \begin{subfigure}[b]{0.45\textwidth}
    \includegraphics[width=\textwidth]{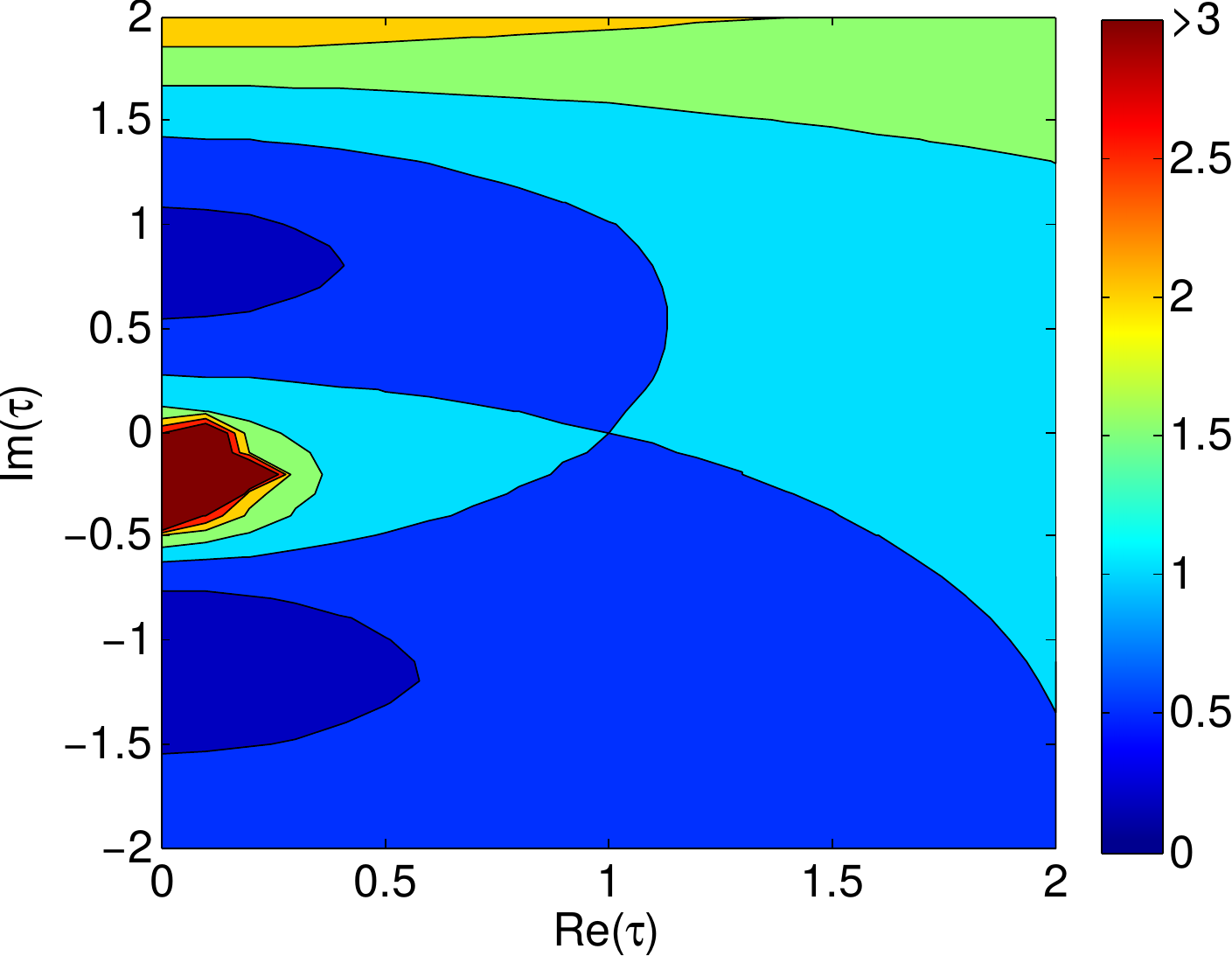}
    \caption{Total error, $p=0$}
    \label{fig:totalerr0}
  \end{subfigure}
  \quad\quad
  \begin{subfigure}[b]{0.45\textwidth}
    \includegraphics[width=\textwidth]{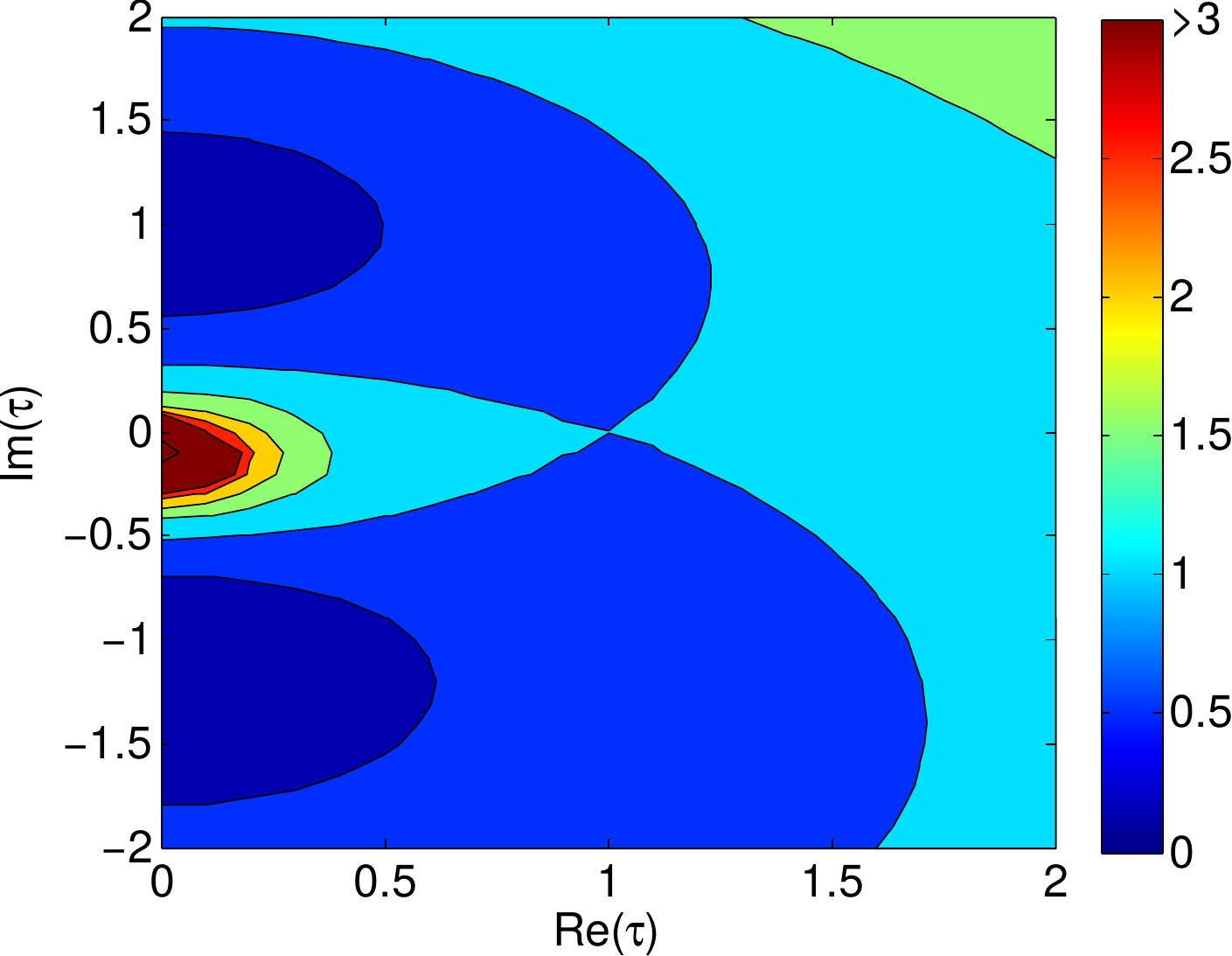}
    \caption{Total error, $p=1$}
    \label{fig:totalerr1}
  \end{subfigure}
  \caption{Normalized dispersive error $\epsilon_{\text{disp}}/\epsilon_{\text{disp}}^1$, 
    dissipative
    error $\epsilon_{\text{dissip}}/\epsilon_{\text{dissip}}^1$, and total error
    $\epsilon_{\text{total}}/\epsilon_{\text{total}}^1$
    for various $\tau\in\mathbb{C}$. Here,
    $k=1$, $h=\pi/4$, and $\epsilon_{\text{disp}}^1, 
\epsilon_{\text{dissip}}^1$ and $ \epsilon_{\text{total}}^1$
denote the errors when $\tau=1$, respectively.}
  \label{fig:errors}
\end{figure}

Results of the dispersion analysis are shown in
Figures~\ref{fig:circle} and~\ref{fig:errors}.  These figures combine
the results from previously discussed $p=0$ case and the $p=1$ cases
to facilitate comparison.  Here, we set $k=1$ and $h=\pi/4$, i.e., $8$
elements per wavelength. Figure~\ref{fig:errors} shows the dispersive,
dissipative, and total errors for various values of
$\tau\in\mathbb{C}$. For both the lowest order and first order cases,
although the dispersive error is minimized at a value of $\tau$ having
nonzero real part, the total error is minimized at a purely imaginary
value of $\tau$. This is attributed to the small dissipative errors
for such $\tau$. Specifically, the total error is minimized when
$\tau=0.87\ii$ in the $p=1$ case. This is close to the optimal value
of $\tau$ found (both analytically and numerically) for $p=0$.  This
value of $\tau$ reduces the total wavenumber error by $90\%$ in the
$p=1$ case, relative to the total error when using $\tau=1$.

\subsection*{Comparison with dispersion relation for the  Hybrid Raviart-Thomas method}

The HRT (Hybrid Raviart-Thomas) method is a classical mixed
method~\cite{ArnolBrezz85,CockbGopal04,RaviaThoma77} which has a
similar stencil pattern, but uses different spaces. Namely, the HRT
method for the Helmholtz equation is defined by exactly the same
equations as~\eqref{eq:globalHelmholtz} but with these choices of
spaces on square elements:
$V(K) = \mathcal{Q}_{p+1,p}(K) \times \mathcal{Q}_{p,p+1}(K)$,
$W(K) = \mathcal{Q}_p(K)$, and $M(F) = \mathcal{P}_p(F)$.  Here
$Q_{l,m}(K)$ denotes the space of polynomials which are of degree at
most $l$ in the first coordinate and of degree at most $m$ in the
second coordinate. The general method of dispersion analysis described
in the previous subsection can be applied for the HRT method. We
proceed to describe our new findings, which in the lowest order case
includes an exact dispersion relation for the HRT method.

\begin{figure}
  \centering
  \begin{subfigure}[b]{0.42\textwidth}
    \includegraphics[width=\textwidth]{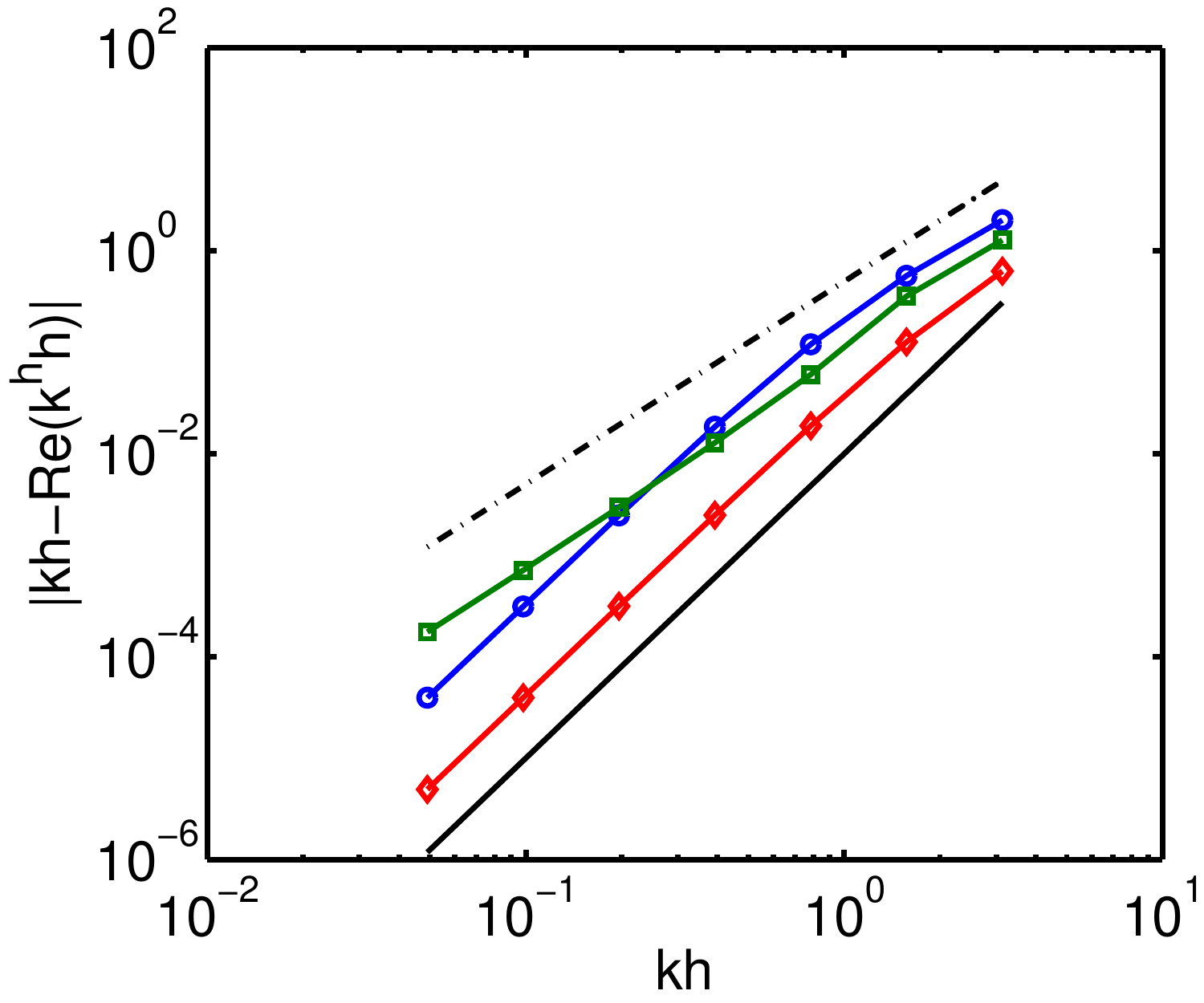}
    \caption{Dispersive error, $p=0$}
    \label{fig:disperrkh0}
  \end{subfigure}%
  \quad\quad
  \begin{subfigure}[b]{0.42\textwidth}
    \includegraphics[width=\textwidth]{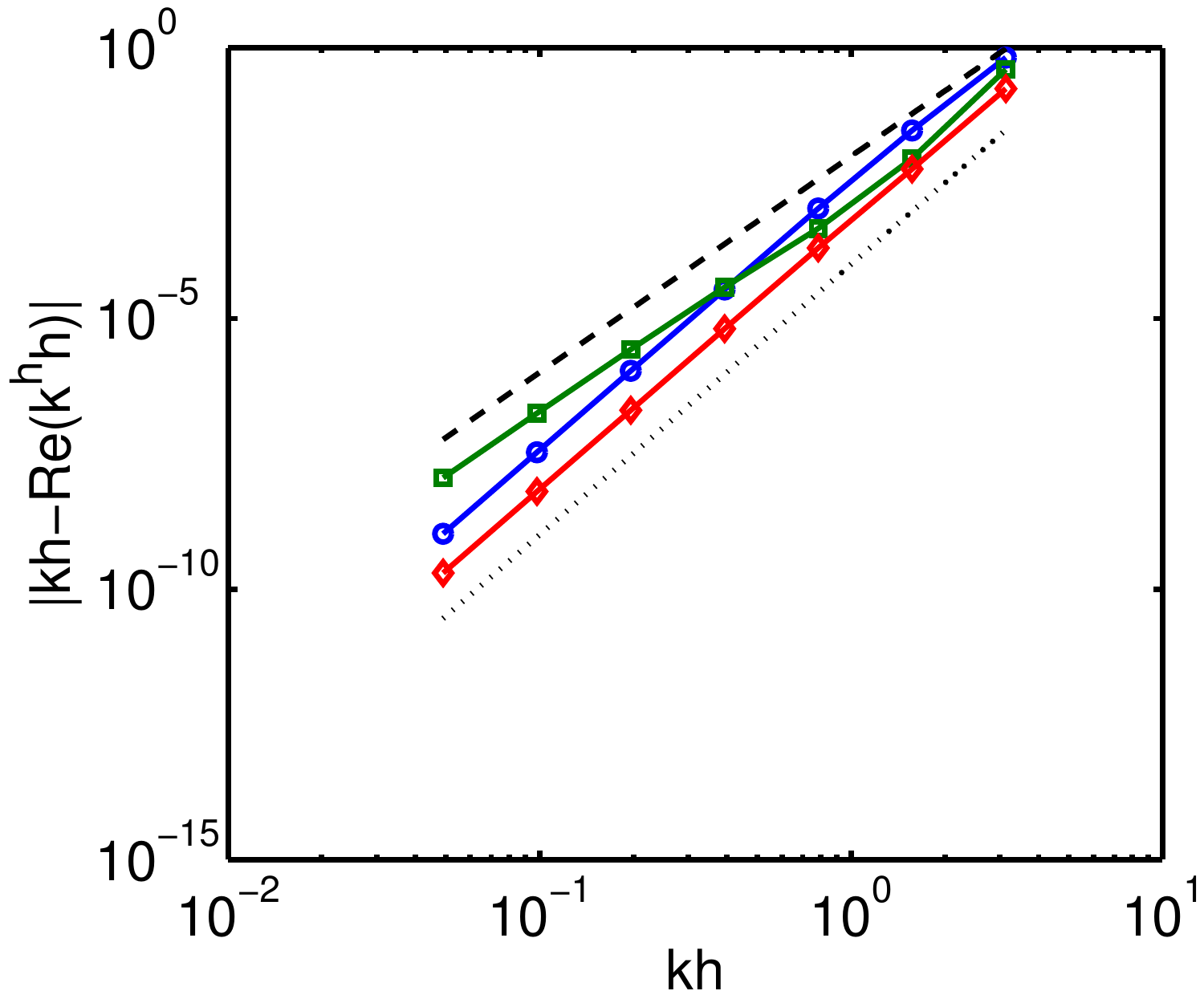}
    \caption{Dispersive error, $p=1$}
    \label{fig:disperrkh1}
  \end{subfigure}%

  \vspace{0.2cm}
  \begin{subfigure}[b]{0.42\textwidth}
    \includegraphics[width=\textwidth]{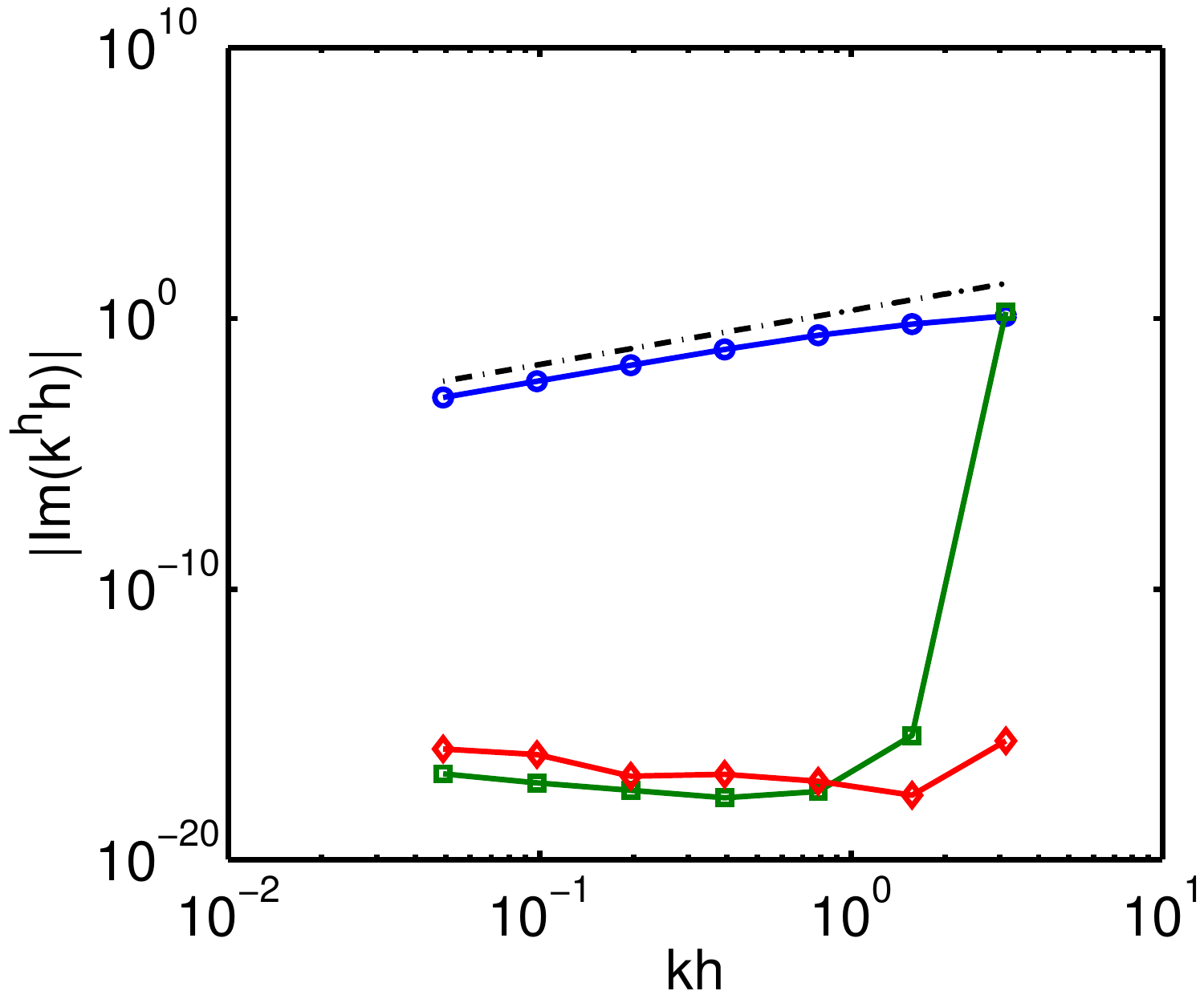}
    \caption{Dissipative error, $p=0$}
    \label{fig:dissiperrkh0}
  \end{subfigure}
  \quad\quad
  \begin{subfigure}[b]{0.42\textwidth}
    \includegraphics[width=\textwidth]{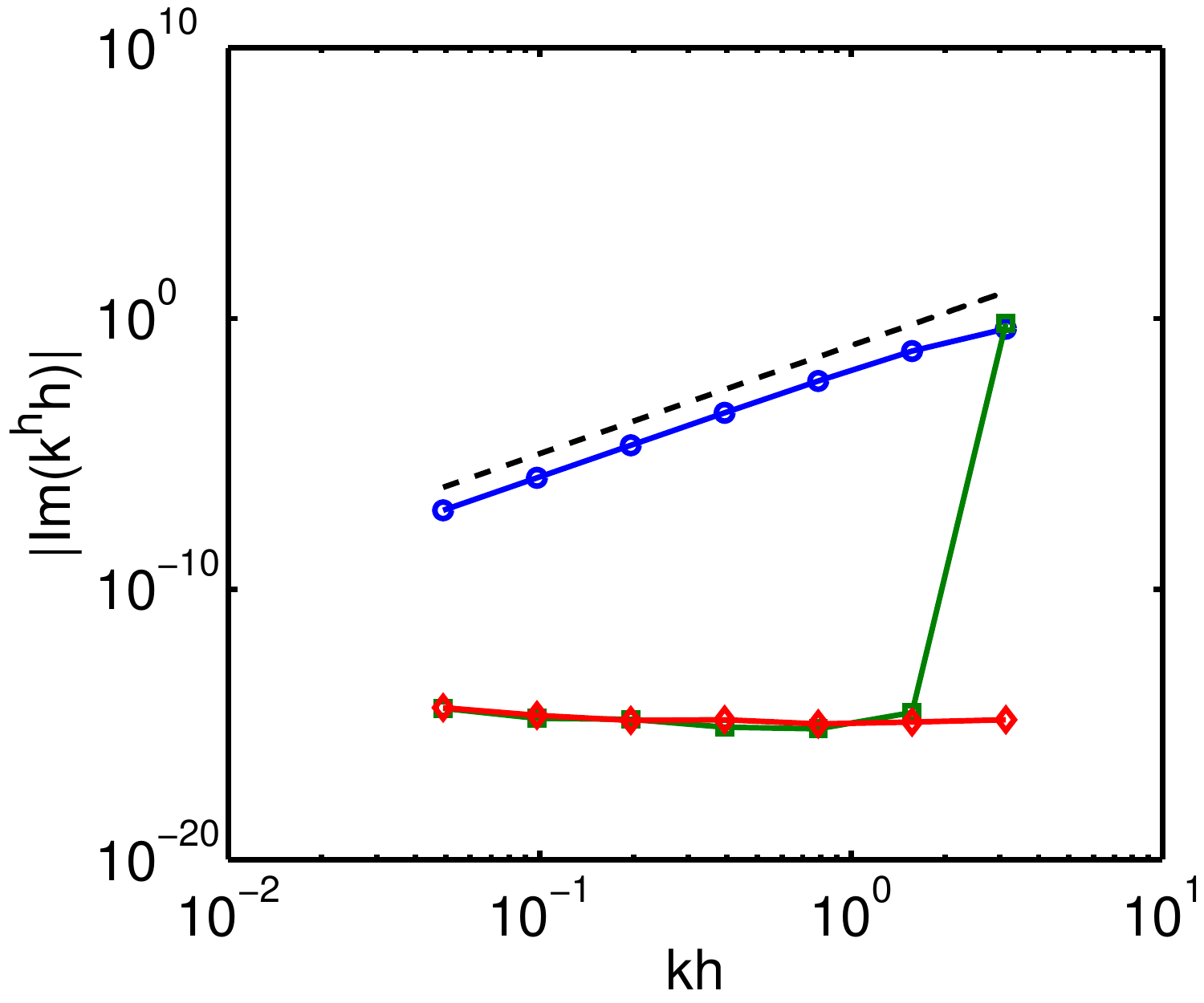}
    \caption{Dissipative error, $p=1$}
    \label{fig:dissiperrkh1}
  \end{subfigure}

  \vspace{0.2cm}
  \begin{subfigure}[b]{0.42\textwidth}
    \includegraphics[width=\textwidth]{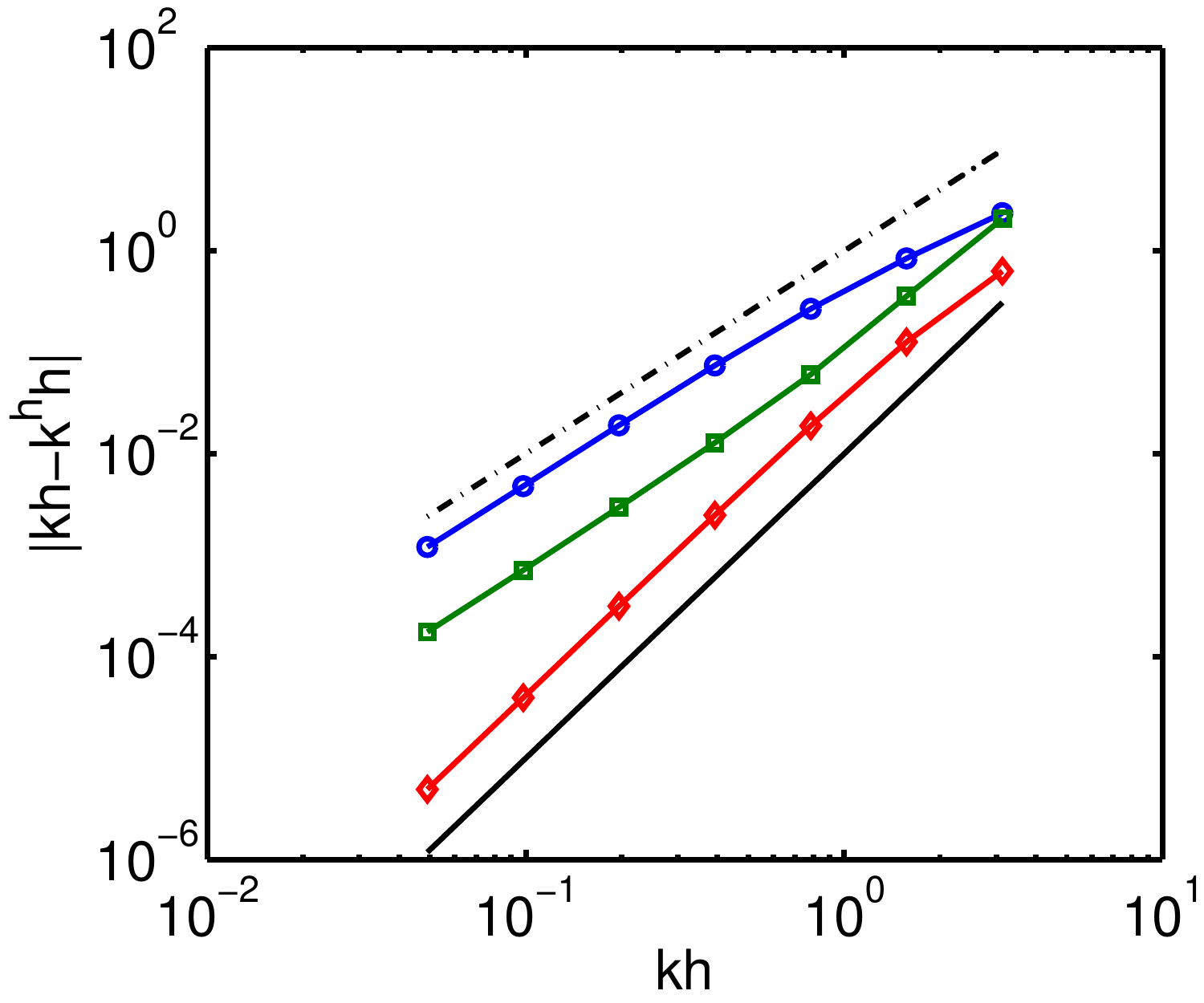}
    \caption{Total error, $p=0$}
    \label{fig:totalerrkh0}
  \end{subfigure}
  \quad\quad
  \begin{subfigure}[b]{0.42\textwidth}
    \includegraphics[width=\textwidth]{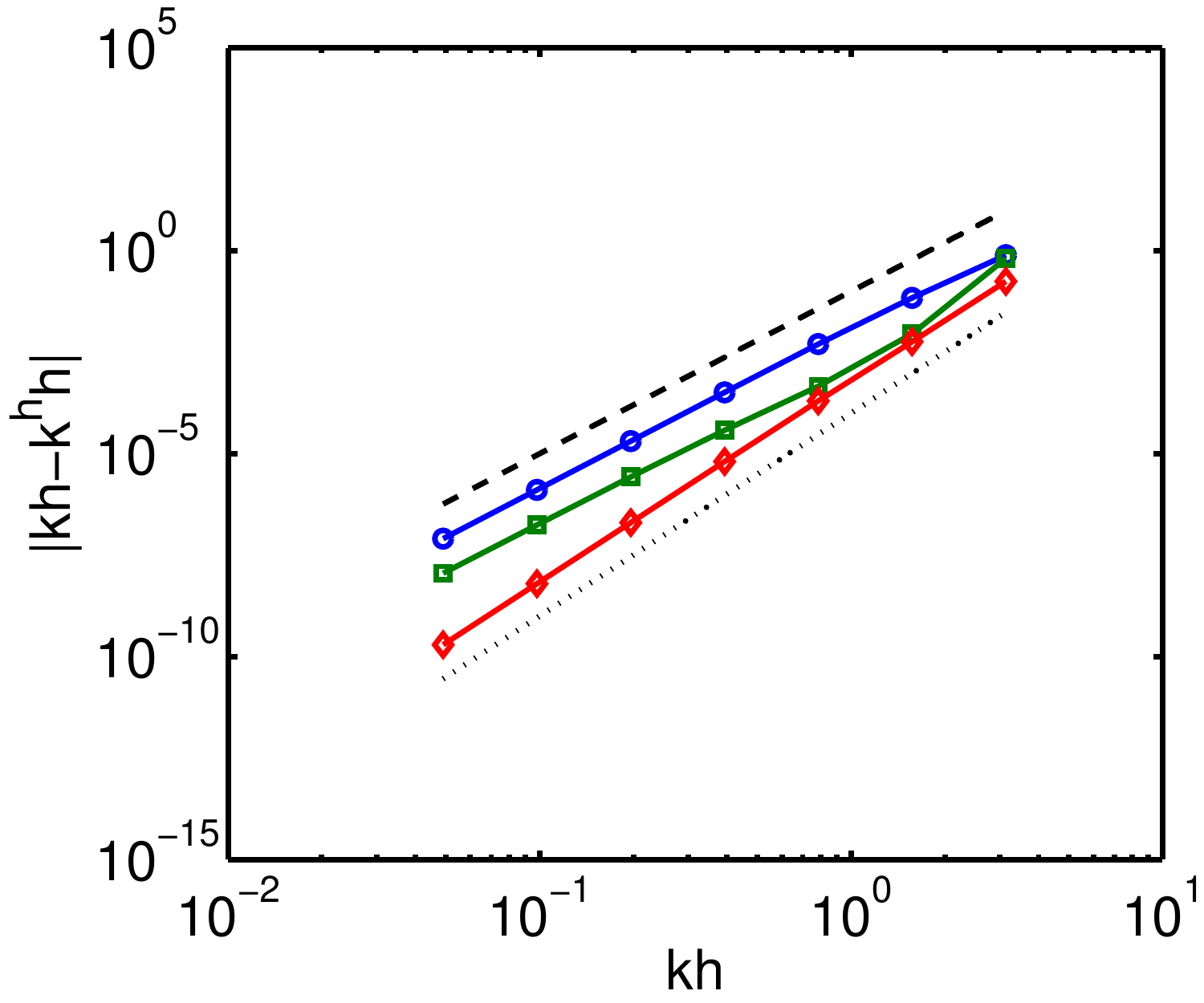}
    \caption{Total error, $p=1$}
    \label{fig:totalerrkh1}
  \end{subfigure}
  \begin{subfigure}[b]{0.7\textwidth}
    \includegraphics[width=\textwidth]{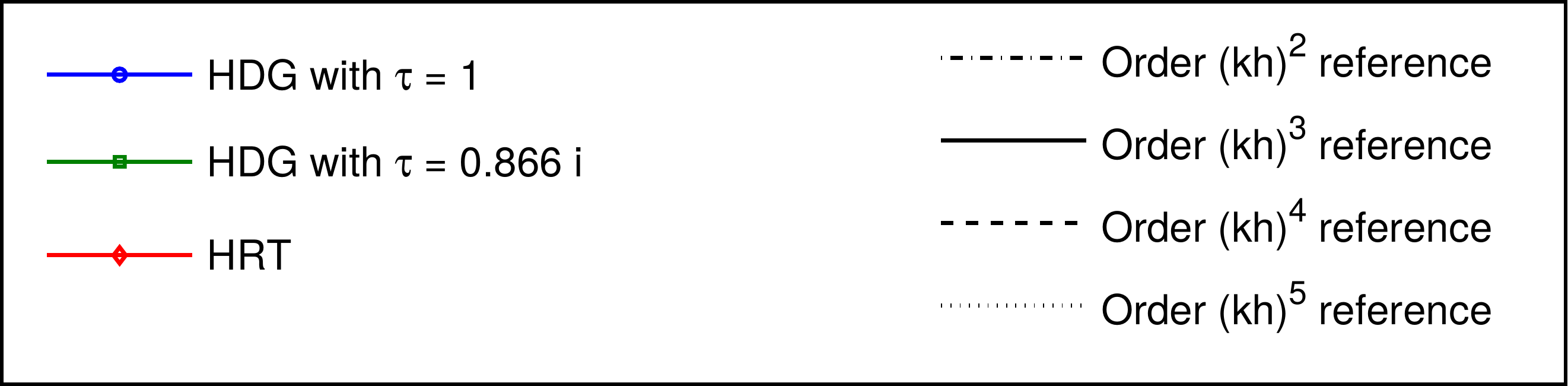}
  \end{subfigure}
  \caption{Convergence rates as $kh \to 0$}
  \label{fig:asymp_errors}
\end{figure}

In the $p=0$ case, after statically condensing the element matrices and
following the procedure leading to~\eqref{eq:2DdispDetF}, we find that
the discrete wavenumber $k^h$ for the HRT method satisfies the 2D
dispersion relation
\begin{equation}
  \label{eq:HRTdisp}
(c_1^2 + c_2^2)
\left(2 (hk)^2 -12\right) + 
c_1^2 c_2^2
\left(4 (hk)^2 + 48 \right)
+ (hk)^2 - 24 =0,
\end{equation}
where $c_j$, as defined in~\eqref{eq:8}, depends on $k^h_j$, which in
turn depends on $k^h$. Similar to the HDG case, we now observe that
the two equations
\begin{equation}
  \label{eq:HRTsuff}
  \left( 2 (hk_j)^2 + 12 \right) c_j^2 + (hk_j)^2 -12 =0,
  \qquad j=1,2,
\end{equation}
are sufficient conditions for~\eqref{eq:HRTdisp} to hold. Indeed, if
$l_j$ is the left hand side above, then
$l_1 (2c_2^2+1) + l_2 (2c_1^2 +1)$ equals the left hand side
of~\eqref{eq:HRTdisp}. The equations of~\eqref{eq:HRTsuff} can
immediately be solved:
\[
hk^h_j = 2\cos^{-1} \left( \frac{12 - (hk_j)^{2}}
{2\,(hk_j)^{2}+12}\right)^{1/2}
\]
Hence, using~\eqref{eq:7} and simplifying using the same type of
asymptotic expansions as the ones we previously used, we obtain 
\begin{equation}
  \label{eq:HRTasymp}
k^hh - kh = -\frac{  \cos( 4\,\theta) +3}{96 } \, (k h)^{3} + O( (kh)^5 )
\end{equation}
as $kh \to 0$. Comparing with~\eqref{eq:9}, we find that in the lowest
order case, {\em the HRT method has an error in wavenumber that is
  asymptotically one order smaller than the HDG method} for any
propagation angle, irrespective of the value of~$\tau$.

To conclude this discussion, we report the results from numerically
solving the nonlinear solution~\eqref{eq:HRTdisp} for $k^h(\theta)$
for an equidistributed set of propagation angles $\theta$. We have
also calculated the analogue of~\eqref{eq:HRTdisp} for the $p=1$ case
(following the procedure described in the previous subsection). Recall
the dispersive, dissipative, and total errors in the wavenumbers, as
defined in~\eqref{eq:10}. After scaling by the mesh size $h$, these
errors for both the HDG and the HRT methods are graphed in
Figure~\ref{fig:asymp_errors} for $p=0$ and $p=1$. We find that the
dispersive errors decrease at the same order for the HRT method and
the HDG method with $\tau=1$. While~\eqref{eq:HRTasymp} suggests that
the dissipative errors for the HRT method should be of higher order,
our numerical results found them to be zero (up to machine accuracy).
The dissipative errors also quickly fell to machine zero for the HDG
method with the previously discussed ``best'' value of
$\tau = \ii\sqrt{3}/2$, as seen from Figure~\ref{fig:asymp_errors}.

\section*{Conclusions}

These are the findings in this paper: 
\begin{enumerate}
\item There are values of stabilization parameters $\tau$ that will
  cause the HDG method to fail in time-harmonic electromagnetic and
  acoustic simulations using complex wavenumbers.
  (See equation~\eqref{eq:1} et~seq.)

\item If the wavenumber $k$ is complex, then choosing $\tau$ so that
  $\re(\tau) \im(k) \le 0$ guarantees that the HDG method is uniquely
  solvable. (See Theorem~\ref{thm:1}.)

\item If the wavenumber $k$ is real, then 
  even when the exact wave problem is not well-posed (such as at a 
  resonance), the HDG method remains uniquely solvable when
  $\re(\tau) \ne 0$. However, in such cases, we found
  the discrete stability to be tenuous. (See
  Figure~\ref{fig:conditioning} and accompanying discussion.)

\item For real wavenumbers $k$, we found that the HDG method
  introduces small amounts of artificial dissipation (see
  equation~\eqref{eq:imkh}) in general. However, when $\tau$ is purely
  imaginary and $kh$ is sufficiently small, artificial dissipation is
  eliminated (see equation~\eqref{eq:11}). 
  In 1D, the optimal values of $\tau$ that asymptotically
  minimize the total error in the wavenumber (that quantifies
  dissipative and dispersive errors together) are $\tau = \pm \ii$
  (see equation~\eqref{eq:kdiff}).

\item In 2D, for real wavenumbers $k$, the best values of $\tau$ are
  dependent on the propagation angle. Overall, values of $\tau$ that
  asymptotically minimize the error in the discrete wavenumber
  (considering all angles) is $\tau = \pm\ii \sqrt{3}/2$ (per
  equation~\eqref{eq:besttau}). While dispersive errors dominate the
  total error for $\tau= \ii\sqrt{3}/2$, dissipative errors dominate
  when $\tau=1$ (see Figure~\ref{fig:asymp_errors}).

\item The HRT method, in both the numerical results and the
  theoretical asymptotic expansions, gave a total error in the
  discrete wavenumber that is asymptotically one order smaller than
  the HDG method. (See~\eqref{eq:HRTasymp} and
  Figure~\ref{fig:asymp_errors}.)

\end{enumerate}




\end{document}